\documentclass[11pt,reqno]{amsart}
\usepackage{amsmath,amssymb}

\usepackage{hyperref}

\usepackage{graphicx,psfrag}

\newcommand{\vectornorm}[1]{\left|\left|#1\right|\right|}
\newcommand{\R}{\mathbb{R}}

\newcommand{\usub}{\underline{u}}
\newcommand{\usup}{\overline{u}}
\def\ds{\displaystyle}

\newtheorem{theorem}{Theorem}[section]
\newtheorem{lemma}[theorem]{Lemma}

\newtheorem{proposition}[theorem]{Proposition}

\newtheorem{remark}[theorem]{Remark}

\newtheorem{definition}[theorem]{Definition}

\numberwithin{equation}{section}

\begin{document}

\title[Saddle-shaped solutions]{Qualitative properties of
saddle-shaped solutions to bistable diffusion equations}

\author{X. Cabr{\'e}}
\thanks{Both authors were supported by the Spain Research projects 
MTM2005-07660-C02-01 and MTM2008-06349-C03-01}
\address{ICREA and Universitat Polit{\`e}cnica de Catalunya,
Departament de Matem{\`a}-tica Aplicada I, Diagonal 647, 08028
Barcelona, Spain}

\email{xavier.cabre@upc.edu}

\author{J. Terra}
\thanks{The second author was supported by the FCT grant  SFRH/BD/8985/2002}
 \address{Universidad de Buenos Aires,
Departamento de Matem\'atica, Ciudad Universitaria, Buenos Aires, Argentina}

\email{jterra@dm.uba.ar}

\begin{abstract}
We consider  the elliptic equation $-\Delta u = f(u)$ 
in the whole $\R^{2m}$, where $f$ is of bistable type. 
It is known that there exists a saddle-shaped 
solution in $\R^{2m}$. This is a solution which changes sign in $\R^{2m}$ and 
vanishes only on the Simons cone ${\mathcal
  C}=\{(x^1,x^2)\in\R^m\times\R^m:  |x^1|=|x^2|\}$.  
It is also known that these solutions are  unstable in dimensions 2 and~4. 

In this article we establish that  when $2m=6$ every saddle-shaped solution
is unstable outside of every compact set and, as a consequence has
infinite Morse index. For this we establish the asymptotic behavior of
saddle-shaped solutions at infinity.
Moreover we prove the existence of a minimal and a maximal
saddle-shaped solutions and derive monotonicity properties for the
maximal solution. 

These results are relevant in connection with
a conjecture of De Giorgi on $1$D symmetry of certain solutions. 
Saddle-shaped solutions are
the simplest candidates, besides $1$D solutions, to be global minimizers in high dimensions,
 a property not yet established.
\end{abstract}

\maketitle

\section{Introduction and main results}

This paper concerns saddle-shaped solutions
to bistable diffusion equations
\begin{equation}\label{eq}
-\Delta u=f(u)\quad {\rm in }\,\R^{n},
\end{equation}
where $n=2m$ is an even integer. It is the follow-up to our previous article~\cite{CT}. 
Here, we study qualitative properties of saddle-shaped solutions,
such as their stability, asymptotic behavior, and monotonicity
properties. 

Our interest on these solutions originates from a conjecture
raised by De Giorgi~\cite{DG2} in $1978$.
It consists of establishing whether every bounded solution
$u$ of the Allen-Cahn equation
\begin{equation}\label{GL}
-\Delta u=u-u^3\quad{\rm in }\,\R^n
\end{equation}
which is monotone in one direction (say, for instance $\partial_{x_n} u>0$
in $\R^n$), 
depends only on one Euclidean variable (equivalently,
all its level sets are hyperplanes), at least if $n\leq 8$.
The conjecture has been proven to be
true when the dimension $n=2$ by Ghoussoub and Gui \cite{GG}, and when $n=3$ by
Ambrosio and Cabr\'e \cite{AC}. For $4\leq n\leq 8$ and assuming the additional
condition
\begin{equation}\label{limiting}
\lim_{x_n\rightarrow\pm\infty}u(x', x_n)=\pm 1
\qquad \text{for all }x'\in \R^{n-1},
 \end{equation}
it has been established by Savin \cite{OS}. 
Recently, del Pino, Kowalczyk, and Wei \cite{dP} have announced 
that the statement of the conjecture does not hold for $n\geq 9$ 
---as suggested in De Giorgi's original statement.
In addition, the monotone and non-flat solution that they construct
satisfies \eqref{limiting}.
However, for $4\leq n\leq 8$ the conjecture in its original statement
is still open, and to our knowledge no clear evidence is known about its validity
or not. That is:

\medskip

\noindent
\textbf{Open Question 1.}
For $4\leq n\leq 8$, does the conjecture hold in its original statement, 
that is, without assuming the limiting condition \eqref{limiting}?

\medskip

Next, we explain how assumption \eqref{limiting} enters in the proof of Savin's result,
and we state another version of the conjecture (Open Question~2 below). 
This new version will lead to the study of saddle-shaped solutions
and to an open problem possibly easier (or more natural) than
Open Question~1 above.
First, recall the following result.

\begin{theorem}[\textbf{Alberti-Ambrosio-Cabr\'e \cite{AAC}}]\label{alba}
Suppose that $u$ is a solution of \eqref{GL} satisfying $\partial_{x_n} u>0$
in $\R^n$ and the condition \eqref{limiting} on limits.
Then, $u$ is a global minimizer in $\R^n$. That is,
$$
{\mathcal E}(u,\Omega)\leq {\mathcal E}(u+\xi,\Omega)
$$ 
for every bounded domain $\Omega$ and every $C^{\infty}$ function
$\xi$ with compact support in $\Omega$. Here, ${\mathcal E}$
denotes the energy functional associated to \eqref{GL}.
\end{theorem}

Recall that the energy functional associated to equation \eqref{eq} is
\begin{equation}\label{energia}
{\mathcal E}(v,\Omega):=\int_\Omega \left\{ \frac{1}{2} |\nabla
v|^2+G(v)\right\} dx,\qquad\text{where } \, G'=-f.
\end{equation}
See \cite{AAC} for the original proof of the Theorem~\ref{alba}
(which was quite involved and used calibrations), and section~3 of
\cite{OS} for a simple proof due to L.~Caffarelli.
Now we can state a deep result of Savin \cite{OS}.

\begin{theorem}[\textbf{Savin \cite{OS}}]\label{teosavin}
Assume that $n\leq 7$ and that $u$ is a global minimizer of \eqref{GL} in $\R^n$. 
Then, the level sets of $u$ are hyperplanes.
\end{theorem}

Note that this result makes no assumptions on the monotonicity or
the limits at infinity of the solution. Now, Savin's result on monotone
solutions needs to assume \eqref{limiting} only to guarantee, by
Theorem~\ref{alba}, that the solution is actually a global minimizer.
Then, Theorem~\ref{teosavin} (and the gain of one dimension
$n=8$ thanks to the monotonicity of the solution)
leads to Savin's result on monotone solutions with limits $\pm 1$.

Again, Theorem~\ref{teosavin} makes no additional assumption on the solutions
(minimizers here). It establishes that in dimensions $n\leq 7$, $1$D solutions
(i.e., solutions depending only on one Euclidean variable)
are the only global minimizers of \eqref{GL}. In addition, the hypothesis 
$n\leq 7$ on its statement is believed to be sharp ---we will 
explain this later in more detail.
That is, in $\R^8$ one expects the existence of a global minimizer
which is not $1$D.

\medskip

\noindent
\textbf{Open Question 2.}
Is there a global minimizer of \eqref{GL}
in $\R^8$ whose level sets are not hyperplanes? Related to this,
it will be natural to ask the following. Are saddle-shaped solutions (as defined below)
global minimizers of \eqref{GL} in $\R^8$, or at least in higher even dimensions?

\medskip

A positive answer
to this last question would give an alternative way 
to that of \cite{dP} to prove the existence of a counter-example
of the conjecture of De Giorgi in $\R^9$. Indeed, saddle-shaped
solutions are even functions of each coordinate $x_i$. Thus,
by a result of Jerison and Monneau~\cite{JM} (Theorem~\ref{JM} below, in the
next section), if a saddle solution
were a global minimizer in $\R^{2m}$, then the conjecture of De Giorgi
on monotone solutions would not hold in $\R^{2m+1}$.

Let us explain why dimension $n=8$, and also saddle-shaped solutions,
play an important role.
By a connection of variational nature between equation~\eqref{GL}
and the theory of minimal surfaces (see \cite{AAC,JM,OS}), every level set
of a global minimizer should converge at infinity to the boundary of a minimal 
set ---minimal here in
in the variational sense, that is, minimizing perimeter. 
See \cite{OS} for precise statements.
Now, a deep theorem (mostly due to Simons~\cite{Si}; see
Theorem~17.3 of~\cite{G})
states that the boundary of
a minimal set in all of $\R^n$ must be a hyperplane if $n\leq 7$.
Instead, in $\R^8$ and higher dimensions, there exist minimal sets different
than half-spaces. The simplest example is the Simons cone,
as proved by Bombieri-De Giorgi-Giusti~\cite{BGG}.

The Simons cone is defined by
\begin{equation}\label{Sim}
{\mathcal C} = \{x\in\R^{2m}: x_1^2 + x_2^2 + \cdots +
x_m^2=x_{m+1}^2 + x_{m+2}^2 + \cdots + x_{2m}^2\}.
\end{equation}
It is easy to
verify that ${\mathcal C}$ has zero mean curvature at every
$x\in{\mathcal C}\backslash\{0\}$, in every dimension $2m\geq 2$.
However, it is only in dimensions $2m\geq 8$ that ${\mathcal C}$ is in addition a
minimizer of the area functional, i.e., it is a minimal cone 
in the variational sense. For all these questions, see the book of Giusti~\cite{G}.
The recent paper \cite{DPP} contains a short proof of the minimality
of the Simons cone when $2m\geq 8$.  
Later in this introduction we will also make some comments on the Morse index of the 
Simons cone depending on the dimension. 

Let us also mention here that for another variational problem (a one-phase
free boundary problem for harmonic functions), a similar program has been undertaken by Caffarelli-Jerison-Kenig~\cite{CJK}
and De~Silva-Jerison~\cite{DSJ}. They have established, respectively, the smoothness
of minimizers in dimension $n\leq 3$ and the existence of a non-smooth
global minimizer in $\R^7$ ---the dimensions in between being 
still an open question.

Saddle-shaped solutions to the bistable diffusion equation are closely related 
to the Simons cone, as follows.
For $x=(x_1,\dots, x_{2m})\in\R^{2m}$, let us define two
radial variables $s$ and $t$ by
\begin{equation}\label{coor}
\left\{\begin{array}{rcll} s& = & {\ds \sqrt{x_1^2+...+x_m^2}}& \geq 0\\
 t & = &{\ds \sqrt{x_{m+1}^2+...+x_{2m}^2}}& \geq 0.
\end{array}
\right.
\end{equation}
The Simons cone is given by 
$$
{\mathcal C}=\{s=t\}=\partial {\mathcal O}, \quad\text{ where }
{\mathcal O}=\{s>t\}.
$$ 
The following is the notion of saddle solution, which we introduced in \cite{CT}. 

\begin{definition}\label{def} 
{\rm Let $f\in C^1(\R)$ be odd. We  say  that 
$u:\R^{2m}\rightarrow\R$ is a} {\it saddle-shaped
solution} {\rm  (or simply a saddle solution) of 
\begin{equation}\label{eq2m}
-\Delta u=f(u) \quad {\rm in }\;\R^{2m}
\end{equation}  if $u$ is a bounded solution of
\eqref{eq2m} and, with $s$ and $t$ defined by \eqref{coor},}
\renewcommand{\labelenumi}{$($\alph{enumi}$)$}
\begin{enumerate}
\item{\rm  $u$ depends only on the variables $s$ and $t$. We write
$u=u(s,t)$; 
\item $u>0$ in ${\mathcal O}:=\{s>t\}$; 
\item  $u(s,t)=-u(t,s)$ in $\R^{2m}$.}
\end{enumerate}
\end{definition}

Saddle-shaped solutions should be relevant in connection with 
Theorem \ref{teosavin} and Open Question 2 above on minimizers of the bistable
diffusion equation due to the different variational
properties of their zero level set (the Simons cone) depending on the dimension
---together with the connection between the diffusion equation and minimal
surfaces. Note also that saddle solutions are
even with respect to each coordinate $x_i$, $1\le i\le 2m$, as
in the result of Jerison-Monneau ---Theorem \ref{JM} below, in section 2.

On the other hand, the conjecture of De Giorgi and Open Question 1 on monotone
solutions are related to minimal graphs ---instead of minimal cones or minimal sets.
The existence of minimal graphs (of functions $\varphi : \R^k \to \R$) 
different than hyperplanes is also well understood. They exist
only when the dimension $k\geq 8$. The simplest one was built by
Bombieri-De Giorgi-Giusti~\cite{BGG} for $k=8$ and has the Simons cone
as zero level set. This minimal graph (living in $\R^9$) is used in
\cite{dP} to construct the counter-example to the conjecture of
De Giorgi in $\R^9$. Note that the Simons cone in $\R^8$ is a variety
of dimension 7, while the previous graph is of dimension 8.

Towards the complete understanding and characterization of global minimizers
(see Open Question 2),
we study saddle-shaped solutions and their qualitative
properties. 
To state our precise results, given a $C^1$ nonlinearity  
$f:\R\rightarrow\R$ and $M>0$, define 
\begin{equation}\label{defG}
G(u)=\int_u^M f.
\end{equation}
We have that
$G\in C^2(\R)$ and $G'=-f$. 
For some $M>0$, and with $G$ defined as above, we assume that
\begin{equation}\label{H}
\left\{\begin{array}{l}
 f \text{ is odd in } \R \\
 G\geq 0=G(\pm M) \text{ in } \R \text{ and } G>0 \text{ in } (-M,M); \\
 f' \text{ is decreasing in } (0,M).
\end{array}\right.
\end{equation}
In Section~2 we comment further these hypothese on $f$.
They are satisfied by $f(u)=u-u^3$, for which
$G(u)=(1/4)(1-u^2)^2$ and $M=1$.

In \cite{CT} we defined saddle-shaped solutions as above
and proved their existence in all even dimensions. 
Namely, we proved:

\begin{theorem}[\cite{CT}]\label{exis} 
Let $f\in C^1(\R)$ satisfy conditions \eqref{H} for some constant $M>0$,
where $G$ is defined by \eqref{defG}. Then, for every even dimension $2m\geq 2$,
there exists a saddle-shaped solution $u\in C^2(\R^{2m})$ 
of $-\Delta u= f(u)$ in~$\R^{2m}$, as in Definition~\ref{def}.
\end{theorem}

Saddle solutions were first studied by Dang, Fife, and 
Peletier~\cite{DFP}  in dimension $n=2$ for $f$ odd, bistable, 
and with $f(u)/u$ decreasing for $u\in(0,1)$. They proved the
existence and uniqueness of a saddle solution in dimension~$2$. 
They also established monotonicity properties
and the asymptotic behavior at infinity of the saddle solution. 
Its instability (see Definition~\ref{def1.1} below), 
already indicated in a partial
result of~\cite{DFP}, was studied in detail by
Schatzman~\cite{Sc} by analysing the linearized 
operator at the saddle solution and showing that, when
$f(u)=u-u^3$, it has exactly one negative eigenvalue. That is,
the saddle solution of the Allen-Cahn equation in dimension~2 
has Morse index~1; see Definition~\ref{morse} below.

The precise notion of stability or instability that we use is the following.

\begin{definition}\label{def1.1}
{\rm Let $f\in C^1(\R)$.}
{\rm 
We say that a bounded solution $u$ of
\eqref{eq} is {\it stable} if the second variation of energy
$\delta^2{\mathcal E}/\delta^2\xi$ with respect to compactly
supported perturbations $\xi$ is nonnegative. That is, if
\begin{equation}\label{stable}
Q_u(\xi):=\int_{{\R}^n}\left\{
|\nabla\xi|^2-f'(u)\xi^2\right\}dx\geq 0 \quad {\rm for \, all}
\;\xi\in C^\infty_c(\R^n).
\end{equation}
We say that $u$ is {\it unstable} if and only if $u$ is not stable.
}
\end{definition}

Clearly, every global minimizer (as defined in Theorem~\ref{alba}) 
is a stable solution.

The instability of the saddle solution in dimension $2$ 
(in the sense of Definition~\ref{def1.1}) is nowadays a
consequence of a more recent result related to the conjecture of De
Giorgi. Namely, \cite{GG} and \cite{BCNf} 
established that, for all $f\in C^1$, every bounded stable solution 
of \eqref{eq} in $\R^2$  must be a $1$D solution, that is, a solution
depending only on one Euclidean variable. 
In particular, the saddle-shaped solution in $\R^2$ can not be stable.

In \cite{CT} we established the instability outside of every compact
set of saddle solutions in dimension $4$ and, as a consequence, their 
infinite Morse index (see Definition~\ref{morse} below). 
In this paper we  establish this same result
in dimension $6$. In addition, the computations in the last section  suggest the
possibility of saddle solutions being stable in dimensions
$2m\geq 8$. Such stability result would be a promising hint towards
the possible global minimality of saddle solutions in 
high dimensions, and hence towards a construction of a counter-example to 
the conjecture of De Giorgi through the method of Jerison-Monneau~\cite{JM}.

The proof of our result in dimension 6 uses two new 
ingredients of independent interest, which hold in any dimension. 
The first concerns the asymptotic behavior of saddle solutions at infinity. 
The second one establishes the existence of
a minimal and a maximal saddle solutions, as well as some key
monotonicity properties of the maximal saddle solution.

Note that for functions $u$ depending only on $s$ and $t$,
such as saddle solutions, the energy functional \eqref{energia} becomes
\begin{equation}\label{enerst}
{\mathcal E}(u,\Omega)=c_m \int_\Omega s^{m-1}t^{m-1}\left\{ \frac{1}{2}
(u_s^2+u_t^2)+G(u)\right\} ds dt,
\end{equation}
where $c_m$ is a positive constant depending only on $m$
---here we have assumed that $\Omega\subset\R^{2m}$ is radially symmetric
in the first $m$ variables and also in the last $m$ variables,
and we have abused notation by identifying $\Omega$ with its
projection in the $(s,t)$ plane. 
In these variables, the semilinear equation \eqref{eq2m} reads
\begin{equation}\label{eqst}
-(u_{ss}+u_{tt})-(m-1){\Big (}\frac{u_s}{s}+\frac{u_t}{t}{\Big
)}=f(u)\qquad \text{for } s>0,\, t>0. 
\end{equation}

The proof of the instability theorem 
in dimension $4$ relied strongly on the following estimate that we established in~\cite{CT} 
(see also Proposition \ref{prop} in section 2). It states that 
\begin{equation}\label{prop1}
|u(x)|\leq \left|u_0\left(\frac{s-t}{\sqrt{2}}\right)\right|
\qquad\text{for all } x \in\R^{2m}
\end{equation}
and for every saddle solution $u$, where $u_0$ is the monotone solution of
$-u''=f(u)$ in $\R$ vanishing at 0. The quantity $|s-t|/\sqrt{2}$ turns out to be 
the distance to the cone ${\mathcal C}$.
This result suggests a new change of variables. Namely we define
\begin{equation}\label{defyz}
\left\{\begin{array}{rcl}
y & = & (s+t)/\sqrt{2}\\
z & = & (s-t)/\sqrt{2},
\end{array}\right.
\end{equation}
which satisfy $y\geq 0$ and $-y\leq z\leq y$. Note that $|z|$ is the distance of any
point $x\in \R^{2m}$ to the Simons cone ${\mathcal C}$; thus, we have
${\mathcal C}=\{z=0\}$ (see Figure \ref{fig1}).

\begin{figure}[htp]
	\begin{center}
		\psfrag{t}{$t \ge 0$}
		\psfrag{s}{$s \ge 0$}
		\psfrag{dy}{$\partial_y$}
		\psfrag{dz}{$\partial_z$}
		\psfrag{C}{${\mathcal C}=\{s=t\}=\{z=0\}$}
		\psfrag{O}{${\mathcal O}=\{s>t\}=\{z>0\}$}
		\includegraphics[width=4cm]{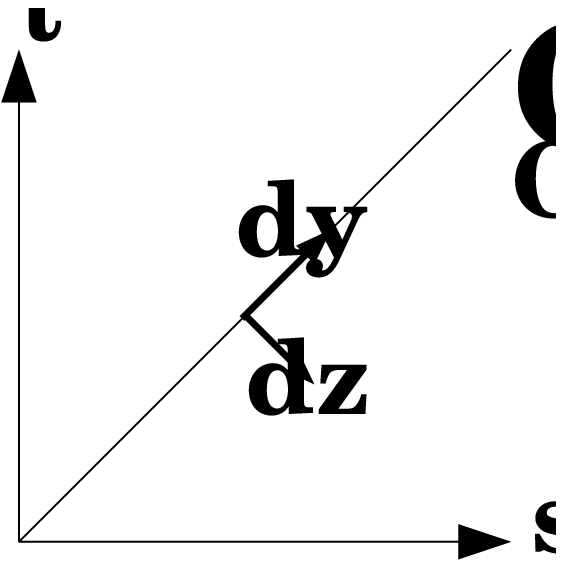}
	\end{center}
	\caption{ }
	\label{fig1}
\end{figure}

With these new variables estimate \eqref{prop1} may be written as 
$$
|u(x)|\leq |u_0(z)|\qquad\text{for all } x \in\R^{2m}.
$$
If we take into account these variables, equation \eqref{eq2m} becomes
\begin{equation}\label{eqyz}
u_{yy}+u_{zz}+\frac{2(m-1)}{y^2-z^2}(yu_y-zu_z)+f(u)=0.
\end{equation}

Using  estimate \eqref{prop1}, instability in dimension
4 follows from establishing that the quadratic form $Q$ defined by
\eqref{stable} with the solution $u$ replaced by 
the explicit function $u_0 ((s-t)/\sqrt 2)$, 
is negative when $n=4$ for some test function $\xi$. By \eqref{prop1}, 
this implies that $Q_u$ is also negative for some 
test function, where $u$ is any given solution vanishing on
${\mathcal C}$. Hence, $u$ is unstable when $n=4$.

This proof can not be generalized to
dimension 6 as it turns out that $u_0((s-t)/\sqrt{2})$ is, in some
sense, asymptotically stable at infinity for 
perturbations with separate variables in $(y,z)$.

Hence, the proof of instability in dimension $6$ requires a more precise argument.
We use the equation satisfied by $\usup_z$, where
$\usup$ is a maximal saddle solution constructed below, as well as 
some monotonicity and
asymptotic properties of  $\usup$ established in this article and
stated in the next two theorems.

The following are
the new results needed in our proof of instability of saddle solutions
in dimension 6. The two results are actually important by
themselves. The first one concerns the asymptotic
behavior at infinity for a class of solutions which contains saddle solutions
---and also other possible solutions not depending only on $s$ and $t$.

\begin{theorem}\label{asym}
Let $f$ satisfy conditions  \eqref{H} and let
$u$ be a bounded solution of $-\Delta u=f(u)$ in $\R^{2m}$ such that $u\equiv 0$
on~${\mathcal C}$, $u>0$ in ${\mathcal O}=\{s>t\}$ and $u$ is odd with
respect to ${\mathcal C}$. Then, denoting 
$$
U(x):=u_0((s-t)/\sqrt{2})=u_0(z) \qquad\text{for }x\in\R^{2m},
$$ 
we have
$$
(u-U)(x)\rightarrow 0 \quad{\rm and } \quad(\nabla u-\nabla
 U)(x)\rightarrow 0,
$$
uniformly as $|x|\rightarrow\infty.$ That is, 
\begin{equation}\label{unif}
\vectornorm{u-U}_{L^{\infty}(\R^{2m}\setminus B_R)}+\vectornorm{\nabla
  u- \nabla U}_{L^{\infty}(\R^{2m}\setminus B_R)}\rightarrow 0\, \text{
  as }\, R\rightarrow\infty.
\end{equation}
\end{theorem}

Our proof of Theorem~\ref{asym} uses a  compactness argument based on 
translations of the solution, combined with
two crucial classification or Liouville type results for monostable equations
in all space and in a half-space.

Theorem \ref{asym} will be used to control some of the integrals
appearing in the proof of instability in $\R^6$. In such proof we will 
establish that the maximal saddle solution $\usup$ is
unstable in dimension 6. The existence of such a maximal saddle
solution, its monotonicity properties, as well as the existence of a minimal saddle
solution $\usub$ are the object of our second result.

\begin{theorem}\label{teominmax} For every nonlinearity $f$
satisfying conditions \eqref{H}, there exist two saddle
solutions $\usub$ and $\usup$ of $-\Delta u=f(u)$ in $\R^{2m}$  which
are minimal and maximal, respectively, in ${\mathcal  O}=\{s>t\}$ 
in the following sense. For every solution $u$
of $-\Delta u=f(u)$ in $\R^{2m}$ vanishing
on the Simons cone and such that $u$ has the same sign as $s-t$,
we have
$$
0 < \usub \leq u\leq \usup\quad {\rm in}\;{\mathcal  O}.
$$ 
As a consequence, we also have
$$
|\usub|\leq |u|\leq |\usup|\quad {\rm in}\;\R^{2m}.
$$

In addition, the maximal solution $\usup$ satisfies:
\renewcommand{\labelenumi}{$($\alph{enumi}$)$}
\begin{enumerate}
\item $-\partial_t \usup \geq 0$ in $\R^{2m}$. Furthermore,
$-\partial_t \usup > 0$ in $\R^{2m}\setminus\{t=0\}$ and 
$-\partial_t \usup = 0$ in $\{t=0\}$.
\item $\partial_s \usup \geq 0$ in $\R^{2m}$. Furthermore,
$\partial_s \usup > 0$ in $\R^{2m}\setminus\{s=0\}$ and 
$\partial_s \usup = 0$ in $\{s=0\}$.
\item As a consequence, $\partial_z \usup > 0$ in $\R^{2m}\setminus\{0\}$; 
recall that
$z=(s-t)/\sqrt{2}$.
\item $\partial_y \usup > 0$ in ${\mathcal O}=\{ s>t\}$; recall that
$y=(s+t)/\sqrt{2}$. As a consequence, 
for every direction $\partial_\eta=\alpha \partial_y-\beta \partial_t$ 
with $\alpha$ and $\beta$
nonnegative constants, $\partial_{\eta} \usup > 0$ in $\{ s>t>0\}$. 
\end{enumerate}
\end{theorem}

It is still an open problem to know if $\usub=\usup$ in dimensions
$2m\geq 4$, which would be equivalent to the uniqueness of saddle
solution. This is only known to hold in dimension $2m=2$ by a result of \cite{DFP}.

The cone of directions of monotonicity in ${\mathcal O}$ described in part (d) of 
the theorem is optimal. Indeed, the level sets of a saddle solution 
(see Figure \ref{fig2}) intersect $\{t=0\}$ orthogonally by regularity
of the solution as a function of the radial variables $s$ and $t$ ---i.e.,
fixed $s$, we must have $u_t=0$ at $\{t=0\}$ since $u(s,\cdot)$ is a $C^1$
radial function. On the other hand, by Theorem~\ref{asym} on the asymptotic behavior
of saddle solutions,
the level sets at infinity become
parallel to the Simons cone at a fixed distance (see also Figure \ref{fig2}).
Thus, the cone of monotonicity in the theorem and figure is optimal.

\begin{figure}[htp]
	\begin{center}
		\psfrag{t}{$t \ge 0$}
		\psfrag{s}{$s \ge 0$}
		\psfrag{C}{${\mathcal C}$}
		\psfrag{u}{$\mu$}
		\psfrag{dy}{$\partial_y$}
		\psfrag{dt}{$-\partial_t$}
		\psfrag{l}{$\{u=\lambda=u_0(\mu)\}$}
		\includegraphics[width=5cm]{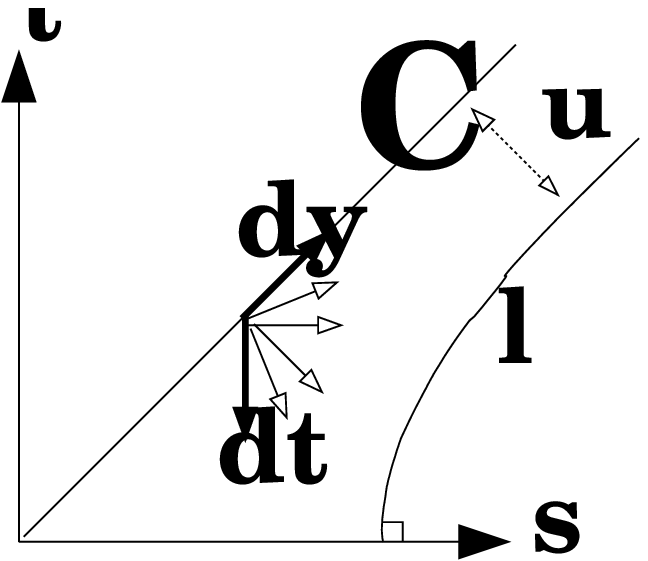}
	\end{center}
	\caption{ }
	\label{fig2}
\end{figure}

\begin{figure}[htp]
	\begin{center}
	        \psfrag{R2}{$\R^2$}
		\psfrag{C}{${\mathcal C}$}
		\psfrag{O}{${\mathcal O}=\{s>t\}=\text{exterior of
		}{\mathcal C}$}
		\includegraphics[width=5cm]{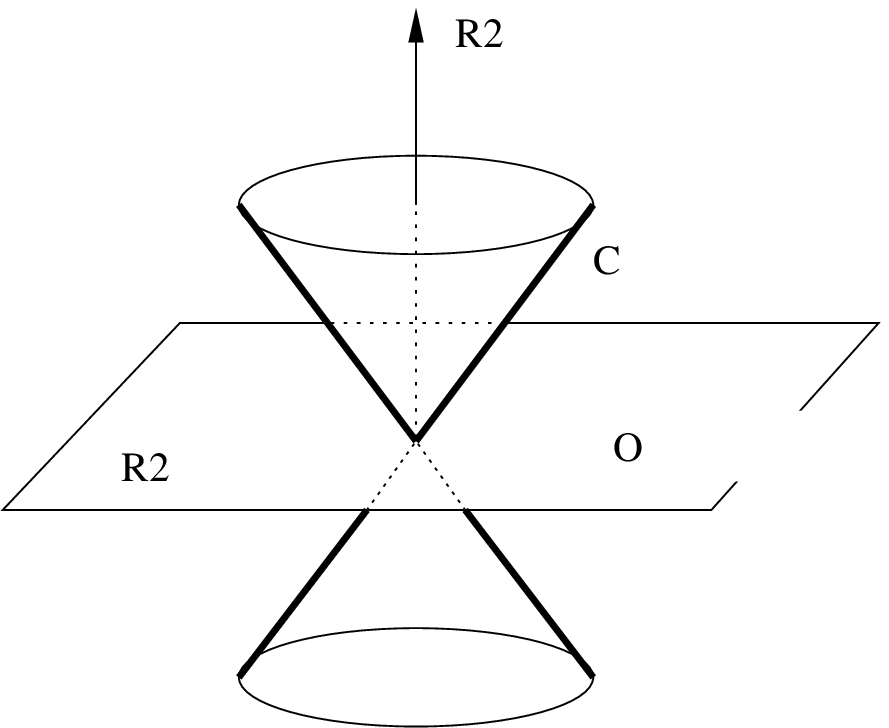}
	\end{center}
	\caption{ }
	\label{fig3}
\end{figure}

We emphasize that our monotonicity results have been achieved without 
using the usual techniques of sliding or moving planes methods. 
The application of such methods to saddle solutions fails in $\R^{2m}$ (for $2m>2$)
due to the particular geometry of saddle solutions  
in dimensions $2m>2$. Indeed, after any translation
in $\R^{2m}$, the exterior ${\mathcal O}$ of the Simons cone does not contain
(neither is contained) in the exterior of the translated cone; see Figure \ref{fig3}.
This fact prevents the use of the sliding or moving planes methods. 
On the other hand, one could think of using these methods in
the simple geometry of $\R^2$ for equation \eqref{eqst} in the $s$ and $t$
variables. But then, for $m>1$, the first order terms in the equation have
the wrong monotonicity to apply any of the two methods.
As we will see, our monotonicity results originates from the stability of saddle
solutions within the region ${\mathcal O}$ ---that is, under perturbations
with compact support in  ${\mathcal O}$.

To state our main theorem in dimension 6, let us recall the notion
of Morse index.

\begin{definition}\label{morse}
{\rm Let $f\in C^1(\R)$.
We say that a bounded solution $u$ of
\eqref{eq}
has {\it finite Morse index} equal to $k\in\{0,1,2,\ldots\}$ if $k$ is the 
maximal dimension of a subspace $X_k$ of $C^1_c (\R^n)$ such that 
$Q_u(\xi) < 0$ for every $\xi\in X_k \setminus \{0\}$. Here
$C^1_c (\R^n)$ is the space of $C^1(\R^n)$ functions with compact support 
and $Q_u$ is defined in \eqref{stable}. 
If there is no such finite integer $k$, we then say that $u$ has
{\it infinite Morse index}.
}
\end{definition}

Every stable solution has finite Morse index equal to $0$. 
It is also easy to verify that every solution with finite Morse
index is stable outside of a compact set (see Theorem~\ref{uns6}
and its proof in section~5 below for more details). 

The following is our main result. It establishes that 
saddle solutions in dimension 6 are unstable outside of every compact set,
and thus have infinite Morse index. It is the analogue of our four dimensional 
result in \cite{CT}. Note that the result applies to a class
of solutions which do not necessarily
depend only on $s$ and $t$ ---and thus 
perhaps larger than the class of saddle solutions.

\begin{theorem}\label{uns6} Let $f$ satisfy conditions  \eqref{H}.
Then, every bounded solution of $-\Delta u=f(u)$ in $\R^6$ such that
$u=0$ on the Simons cone ${\mathcal C}=\{s=t\}$ and $u$ is positive in
${\mathcal O}=\{s>t\}$ is unstable. Furthermore, every such solution
$u$ is unstable outside of every
compact set. That is, for every compact set $K$ of $\R^6$
there exists $\xi\in C^1(\R^6)$ with compact support in
$\R^6\setminus K$ for which 
$Q_u(\xi) < 0$, where $Q_u$ is defined in \eqref{stable}. 
As a consequence, $u$ has infinite Morse index in the sense of
Definition~$\ref{morse}$.

In particular, every saddle
solution as in Definition $\ref{def}$ is unstable outside of
every compact set if $2m=6$.
\end{theorem}

To establish this result, we prove
that $\usup$ is unstable outside of every compact set by constructing 
test functions
$$
\xi(y,z)=\eta(y)\usup_z(y,z)
$$ 
such that $Q_{\usup}(\xi)<0$.
We need to use the asymptotic convergence and monotonicity results for $\usup$
of Theorems $\ref{asym}$ and $\ref{teominmax}$.
Since $\usup$ is maximal, its instability outside of compact sets 
implies that this same instability 
property holds for all bounded solutions $u$ vanishing on the Simons cone
${\mathcal C}$ and positive in ${\mathcal O}$.

Let us comment on the Morse index 
of stationary surfaces, i.e., surfaces of zero mean curvature. 
The usual proof of 
the instability of the Simons cone in dimension 4 and 6 (see \cite{G}) 
also leads to its instability outside of every compact set, and hence
to its infinite Morse index property. A precise study of the Morse index of 
stationary surfaces close to the Simons cone is made in~\cite{A2}
through the analysis of intersection numbers.
Note that in dimension 2, outside of a ball centered at the origin
the Simons cone consists 
of 4 disconnected half-lines ---a stable configuration.
Note the analogy of these results with those for saddle solutions: 
we have proved that they have infinite Morse index in dimensions
4 and 6, while Schatzman established that in $\R^2$ the saddle solution has Morse
index 1.

The instability
of saddle solutions in low dimensions is related to
perturbations which do not vanish on the
Simons cone, and hence, which change the zero level set of the solution.
Indeed, the following result states that in all dimensions, every
solution that vanishes on the Simons
cone ${\mathcal C}$ and has the same sign as $s-t$ is stable
under perturbations with compact support which vanish on the
Simons cone.

\begin{proposition}\label{cstable}
Let f satisfy conditions \eqref{H}. Then, every bounded  solution $u$ 
of $-\Delta u=f(u)$ in $\R^{2m}$ that vanishes on the Simons cone
${\mathcal C}=\{s=t\}$ and has the same sign as $s-t$,  is stable in $\R^{2m}$ with
respect to perturbations vanishing on ${\mathcal C}$. That is,
$Q_u(\xi) \geq 0$ for all
$\xi\in C^1(\R^{2m})$ with compact support and such that $\xi \equiv 0$ in
${\mathcal C}$, where $Q_u$ is defined in \eqref{stable}.
\end{proposition}

\begin{remark}
{\rm
The last section shows that the maximal solution $\usup$ is in some sense
asymptotically stable at infinity in dimensions $2m\geq 8$. This indicates
that saddle-shaped solutions might be stable in dimensions $2m\geq 8$.
}
\end{remark}

The recent work by Alessio, Calamai, and Montecchiari~\cite{ACM} concerns scalar saddle type solutions in
$\R^2$ changing sign on more nodal lines than $x_1=\pm x_2$, whereas Alama, 
Bronsard, and Gui in \cite{ABG} studied vector-valued saddle solutions in $\R^2$.

The paper is organized as follows. In section $2$ we present some of the 
results mentioned in this introductions and used later in the paper.
Section 3 concerns the maximal and minimal saddle solutions and their properties;  
we prove Theorem \ref{teominmax} and Proposition
\ref{cstable}. Section 4 is devoted to the study of the asymptotic
behavior of saddle shaped solutions; we establish Theorem~\ref{asym}.
Section 5 contains the proof of instability of saddle solutions
in $\R^6$, Theorem~\ref{uns6}.
Finally, in section 6 we show that the maximal saddle solution $\usup$ is
in some sense asymptotically stable at infinity in dimensions $2m\geq 8$.

\section{Preliminaries}

This section contains a more detailed description and statement of 
some of the results mentioned in the previous section. Most of these results 
will be used throughout the paper.

We begin with the existence of a solution in dimension one.
The proof of this lemma, which follows from
integrating the ODE $\ddot{u}-G'(u)=0$, can be found in \cite{AC}
---see also a sketch of the proof below, after the statement.

\begin{lemma}[see e.g. \cite{AC}] \label{lemma1D} 
Let $G\in C^2(\R)$.
There exists a bounded function $u_0\in C^2(\R)$ satisfying
$$\ddot{u}_0-G'(u_0)=0 \quad \hbox{and} \quad \dot{u}_0>0 
\quad \hbox{ in } \R
$$
if and only if there exist two real numbers $m_1<m_2$ for which
$G$ satisfies
\begin{equation}\label{G1}
G'(m_1)=G'(m_2)=0 \qquad\text{ and }
\end{equation}
\begin{equation}\label{G2}
G>G(m_1)=G(m_2) \quad \hbox{in } (m_1,m_2).
\end{equation}
In such case we have $m_1=\lim_{\tau\rightarrow -\infty} u_0(\tau)$ and 
$m_2=\lim_{\tau\rightarrow +\infty} u_0(\tau)$. Moreover, the solution
$u_0=u_0(\tau)$ is unique up to translations of the independent 
variable $\tau$. 

Adding a
constant to $G$, assume that
\begin{equation}\label{Gzero}
G(m_1)=G(m_2)=0.
\end{equation}
Then, we have that
\begin{equation}\label{hamilt}
\frac{\dot{u}_0^2}{2}= G(u_0) \qquad\text{in }\R .
\end{equation}
If in addition 
$G''(m_1)\neq 0$ and $G''(m_2)\neq 0$, then
\begin{equation}\label{cresc1D}
0<\dot{u}_0(\tau)\leq C e^{-c|\tau|} \quad \text{in }\R
\end{equation}
for some positive constants $C$ and $c$, and
\begin{equation}\label{int1D}
\int_{-\infty}^{+\infty} \left\{ \frac{1}{2}\dot{u}_0(\tau)^2 +
G(u_0(\tau)) \right\} d\tau <+\infty .
\end{equation} 
\end{lemma}

Given G satisfying ($\ref{G1}$), ($\ref{G2}$), and \eqref{Gzero},
to construct $u_0$ 
we simply choose any $m_0\in(m_1,m_2)$ and define
$$\phi(\sigma)=\int_{m_0}^{\sigma}\frac{dw}{\sqrt{2(G(w))}} \quad\text{ for }
\sigma\in (m_1,m_2).$$
Then let $u_0:=\phi^{-1}$ be the inverse function of $\phi$.
This formula is found multiplying $\ddot{u}-G'(u)=0$ by $\dot{u}$
and integrating the equation ---which also gives the necessity
of conditions ($\ref{G1}$) and ($\ref{G2}$) for existence.
The above definition of $u_0$ leads automatically to \eqref{hamilt}.
Under the additional hypothesis
$G''(m_i)\neq 0$, $G$ behaves like a quadratic function
near each $m_i$. Using the expression above for $\phi$, this gives
that $\phi$ blows-up logarithmically at $m_i$, and thus 
its inverse function $u_0$
attains its limits $m_i$ at $\pm\infty$ exponentially. 
From this and identity \eqref{hamilt}, the exponential decay
\eqref{cresc1D} for $\dot{u}_0$, as well as \eqref{int1D}, follow.

In relation with assumptions \eqref{G1} and \eqref{G2} on $G$, our hypothesis
\begin{equation}\label{condG}
G\geq 0=G(\pm M) \; {\rm in} \,\R \quad{\rm and}\quad G>0 \;{\rm in}\,
(-M,M)
\end{equation}
appearing in \eqref{H}, guarantees  the existence of an
increasing solution $u_0$ of \eqref{eq} in dimension 1, that is 
in all of $\R$, taking values onto
$(-M,M)$, as stated in Lemma \ref{lemma1D}.
The nonlinearities $f$ that satisfy property \eqref{condG} 
are of ``balanced bistable type", while the potentials $G$
are of ``double-well type". 
Typical examples are the Allen-Cahn (or scalar Ginzburg-Landau)
equation $-\Delta u= u-u^3$, where $G(u)=(1/4)(1-u^2)^2$ and
$M=1$, and also the equation $-\Delta u=\sin (\pi u)$, for which $M=1$ and
$G(u)=(1/\pi)(1+\cos (\pi u))$.

Since the increasing solution whose existence is established above 
is unique up to translations of the
independent variable, we normalize it to vanish at the origin and
we call it~$u_0$. Thus, we have
\begin{equation}\label{u0}
\left\{
\begin{array}{l}
u_0:\R\rightarrow (-M,M) \\
u_0(0)=0,\, \dot{u}_0>0 \ {\rm in}\, \R, \;{\rm and}\\
 -\ddot{u}_0=f(u_0) \quad {\rm in}\, \R.
 \end{array}
 \right.
\end{equation}
For the Allen-Cahn nonlinearity $f(u)=u-u^3$, 
the solution $u_0$ can be computed explicitly and it is given by 
$u_0(\tau)=\tanh(\tau/\sqrt{2})$.

We can now consider the family of
$1$D solutions to \eqref{eq} in $\R^n$, given by
\begin{equation}\label{1d}
u_{b,c}(x)=u_0(b\cdot x+c)\quad \textrm{ for }  x\in\R^n,
\end{equation}
for every given $b\in\R^n$ with $|b|=1$ and $c\in\R$.
Under hypothesis ($\ref{condG}$) on the nonlinearity, 
every $1$D solution $u_{b,c}$ is a
global minimizer of \eqref{eq}, by Theorem~\ref{alba} in the introduction when
$f(u)=u-u^3$ (or by the more general result in~\cite{AAC}). 
In particular, $u_{b,c}$ is a stable
solution. 
By Savin's result, Theorem~\ref{teosavin} above, we know
that $1$D solutions are the only global minimizers of 
the Allen-Cahn equation \eqref{GL} if $n\leq 7$.

In $\R^8$ it is expected that global minimizers
which are not $1$D exist. As argued in the introduction, natural 
candidates to be minimizers of this
type are saddle-shaped solutions.

In \cite{CT} we established the existence of saddle solutions to \eqref{eq} 
(see Theorem \ref{exis} in the Introduction).
By classical elliptic regularity theory, it is well known that
for $f\in C^1(\R)$, every bounded solution of $-\Delta u=f(u)$ in $\R^{n}$
satisfies $u\in C^{2,\alpha} (\R^n)$ for all $0<\alpha <1$, and
thus it is a classical solution. In particular, saddle solutions
are classical solutions.

Moreover (see \cite{CT}), there exists a saddle solution 
$u$ satisfying $|u|<M$ in $\R^{2m}$, as well as the energy estimate
\begin{equation}\label{est-ener}
{\mathcal E}(u,B_R)=\int_{B_R} \left\{ \frac{1}{2}|\nabla
u|^2+G(u)\right\} dx \leq CR^{2m-1} \qquad\text{for all }R>1,
\end{equation}
where $C$ is a constant independent of $R$ and $B_R$ denotes the
open ball of radius~$R$ centered at $0$. Both these estimates as well as existence itself are establish assuming only the first two conditions on $f$ 
in \eqref{H}.

A crucial ingredient that we used in the proof of the instability theorem
in dimension $4$ and which we will use also in this paper 
is the following pointwise estimate.

\begin{proposition}[\textbf{\cite{CT}}]\label{prop} 
Let $f\in C^1(\R)$ satisfy \eqref{H}. 
If $u$ is a bounded solution of $-\Delta u=f(u)$ in $\R^{2m}$ that
vanishes on the Simons cone ${\mathcal C}=\{s=t\}$, then
\begin{equation}\label{estimate}
|u(x)|\leq |u_0({\rm dist}(x,{\mathcal C}))|=\left| u_0{\Big
(}\frac{s-t}{\sqrt{2}}{\Big )}\right| \quad \text{ for all }
x\in\R^{2m},
\end{equation}
where $u_0$ is defined by \eqref{u0} and 
${\rm dist}(\cdot,{\mathcal C})$ denotes the distance 
to the Simons cone, which is equal to $|s-t|/\sqrt{2}$.

In addition, the function $u_0((s-t)/\sqrt{2})$ is a 
supersolution of $-\Delta u=f(u)$ in the set
${\mathcal O}=\{s>t\}$.

\end{proposition}

Estimate \eqref{estimate} follows easily from an important estimate
for global solutions of semilinear equations, called Modica estimate~\cite{M1}.

Finally we recall here the statement of the theorem of Jerison and Monneau 
\cite{JM} establishing that  the existence of a bounded, 
even with respect to each coordinate, global minimizer of \eqref{GL} in 
$\R^{n-1}$ would yield the existence of a counter-example to the conjecture 
of De Giorgi in $\R^{n}$.

\begin{theorem}[\textbf{Jerison-Monneau} \cite{JM}]\label{JM}
Let $G\in C^{2}(\R)$ satisfy
$$
G\geq 0\ {\rm in }\;\R ,\quad
G>0\ {\rm in }\; (-1,1),\ {\rm and }\quad G(1)=G(-1)=0.$$
Assume that there
exists a global minimizer $v$ of $\Delta v=G'(v)$  on $\R^{n-1}$ such that $|v|<1$ and
$v$ is even with respect to each coordinate $x_i,1\leq
i\leq n-1$. Then,
for each $\gamma\in(0,\sqrt{2G(v(0))})$  there exists a solution
$u\in C^2(\R^n)$ to $$\Delta u=G'(u) \quad \text{ in }\R^n$$ satisfying
$$|u|\leq 1\quad {\rm and }\quad \partial_{x_n}u>0\quad {\rm in
 }\;\R^n,$$
 and such that, for one
$\lambda\in\R$, the set $\{u=\lambda\}$ is not a hyperplane.

Moreover, this solution $u$ is a global minimizer in $\R^n$,
it is even in the first $n-1$ coordinates, and
satisfies $\partial_{x_n}u(0)=\gamma$ and $u(0)=v(0)$.
\end{theorem}

\section{Minimal and maximal saddle solutions}\label{sctminmax}

As in the work of Dang, Fife, and Peletier \cite{DFP} in dimension 2, in this
section we show the existence in higher dimensions
of a minimal, $\usub$, and a maximal,
$\usup$,  saddle solutions and we
prove the monotonicity properties of $\usup$ stated in
Theorem \ref{teominmax}. We recall that we establish, among other properties,  
that the maximal solution is increasing in the $z$ direction, fact that we will 
use in the proof of instability of saddle solutions
in dimension $6$ (see section~\ref{sctuns6}).

Moreover, we prove Proposition~\ref{cstable} stating the stability 
under perturbations vanishing on the Simons cone of saddle solutions.

The proof of Theorem \ref{teominmax} and of Proposition
\ref{cstable} are given at the end of subsection \ref{susct342}.

\subsection{Existence of minimal and maximal solutions}\label{susct341}

For $R>0$, consider the open sets
\begin{equation*}
T_R=\{x\in\R^{2m}: t<s<R\}.
\end{equation*}
and
\begin{equation*}
O_R={\mathcal O}\cap B_R=\{x\in\R^{2m}:  s>t\} \cap B_R,
\end{equation*}
where $B_R$ is the open ball in $\R^{2m}$ of radius $R$ centered at $0$. 
Note that $T_R\supset {\mathcal O}_R$.

In the next lemmas, we will need the following simple facts on the
nonlinearity.
 
\begin{remark}\label{f'dec}
{\rm
Let $f$ satisfy assumptions \eqref{H}. Define
\begin{equation}\label{defg}
g(\rho):=f(\rho)-f'(M)\rho \qquad\text{for } 0\leq\rho\leq M. 
\end{equation}
Then,  $g$ is positive and increasing in $(0,M)$. Indeed,
since $f'$ is decreasing in $(0,M)$,
$g'(\rho)=f'(\rho)-f'(M)>0$ in $(0,M)$.
That is, $g$ is increasing in $(0,M)$. Since $g(0)=0$, we deduce
that $g>0$ in $(0,M)$.

Another fact that we will use is that
$f(\rho)/\rho$ is
decreasing in $(0,M)$. Indeed, given  $0<\rho<M$ 
there exists $\rho_1$ with $0<\rho_1<\rho$ and 
\begin{equation}\label{frho}
\frac{f(\rho)}{\rho}=\frac{f(\rho)-f(0)}{\rho -0}=f'(\rho_1)>f'(\rho),
\end{equation} 
since $f'$ is decreasing in $(0,M)$.
Therefore,
$$\left(\frac{f(\rho)}{\rho}\right)'=\frac{f'(\rho)\rho-f(\rho)}{\rho^2}=
\frac{f'(\rho)-f'(\rho_1)}{\rho}<0.$$
}
\end{remark}

The following definition and remark on stability of positive solutions
will be used in some of the next lemmas, as well as in the proof of
Proposition~\ref{cstable}.

\begin{definition}\label{defstablebounded}
{\rm
Let $\Omega\subset\R^n$ be an open set.
We say that a bounded solution $u$ of $-\Delta u = f (u)$ in $\Omega$
is {\it semi-stable} in $\Omega$ if the second
variation of energy $\delta^2{\mathcal E}/\delta^2\xi$ with respect to
perturbations $\xi\in H_0^1(\Omega)$ is nonnegative. 
That is, if 
\begin{equation}\label{semist}
Q_u(\xi):=\int_{\Omega}\left\{
|\nabla\xi|^2-f'(u)\xi^2\right\}dx \geq  0 \quad \text{ for  all }
\xi\in H_0^1(\Omega).
\end{equation}
}
\end{definition}

\begin{remark}\label{stability}
{\rm
Let $f$ satisfy assumptions \eqref{H} and $\Omega\subset\R^n$ be an open set
(bounded or unbounded).
Let $u$ 
be a positive solution of $-\Delta u = f (u)$ in $\Omega$
such that $0<u<M$ in $\Omega$.
Then, $u$ is a semi-stable solution in $\Omega$.

The proof of this fact is simple. 
By \eqref{frho}, $f'(w)\leq f(w)/w$ for all 
real numbers $w\in (0,M)$. Hence we have
\begin{equation}\label{possuper}
-\Delta u =f(u) \geq f'(u) u \qquad\text{in } \Omega.
\end{equation}
That is, $u$ is a positive supersolution of the linearized
problem $-\Delta - f'(u)$ at $u$ in all of $\Omega$.
We claim that, as a consequence, the quadratic form $Q_u(\xi)$ 
is nonnegative for all $\xi\in C^1$ with compact support in 
$\Omega$. By approximation, the same
holds for all $\xi\in H^1_0(\Omega)$.

This claim can be proved in two different ways. First, by a simple 
integration by parts argument. It consists of taking any
$\xi\in C^1$ with compact support in $\Omega$, multiplying \eqref{possuper}
by $\xi^2/u$, and 
integrating by parts to get
$$
\int_{\Omega} f'(u)\xi^2 = \int_{\Omega} f'(u) u \frac{\xi^2}{u} \leq 
\int_{\Omega} (-\Delta u) \frac{\xi^2}{u} \leq \int_{\Omega} 
\nabla u \nabla\xi \frac{2\xi}{u} -\int_{\Omega} \frac{|\nabla u|^2}{u^2} \xi^2.
$$
Now, using the Cauchy-Schwarz inequality, we are led to \eqref{semist}.

Another proof of the claim is the following. Since the linearized
operator has a positive supersolution, its principal eigenvalue
is nonnegative (see \cite{BNV}). Since the principal eigenvalue
coincides with the first eigenvalue, Rayleigh criterion gives \eqref{semist}.
}
\end{remark}

Our first lemma concerns the existence of a maximal solution in $T_R$
and some of its properties.

\begin{lemma}\label{maxR}
Let $f$ satisfy conditions \eqref{H}. Then, there exists a positive
solution $\usup_R$ of 
\begin{equation}\label{exisusupsta}
\left\{
\begin{array}{rcll}
-\Delta \usup_R & = & f(\usup_R) &
  \text{ in } T_R:=\{ t<s<R\}\\
\usup_R & = & u_0(z) &  \text{ on }\; \partial T_R,
\end{array}
\right.
\end{equation}
with $0<\usup_R<M$ in $T_R$, 
which is maximal in the following sense. We have that $\usup_R\geq u$ in $T_R$ 
for every positive solution $u$ of $-\Delta u=f(u)$ in $T_R$ satisfying  
$u \leq u_0(z)$ in all of $\overline{T_R}$.

Moreover, $\usup_R$ is semi-stable in $T_R$ in the sense of 
Definition~\ref{defstablebounded}, and $\usup_R$ depends only on $s$ and $t$.
\end{lemma}

\begin{proof}
Let  $u_0$ be the  solution in $\R$ defined
by \eqref{u0} and $z=(s-t)/\sqrt{2}$. With $g$ as defined in
\eqref{defg}, we write equation $-\Delta u = f(u)$ as
$\{ -\Delta -f'(M) \}u = g(u)$.

We use the method of monotone iteration.
Define a sequence of functions $\usup_{R,k}$ by $\usup_{R,0}(x)  =  u_0(z)$ and
by solving the linear problems
\begin{equation}\label{exisusup}
\left\{
\begin{array}{rcll}
\{ -\Delta -f'(M) \}\usup_{R,k+1} & = & g(\usup_{R,k}) &
  \text{ in } T_R\\
\usup_{R,k} & = & u_0(z) &  \text{ on }\; \partial T_R.
\end{array}
\right.
\end{equation}
Since $-\Delta-f'(M)$ is obtained by adding a positive constant to $-\Delta$,
it satisfies the maximum principle and hence the above problem admits
a unique solution $\usup_{R,k+1}=\usup_{R,k+1}(x)$. Furthermore (and
here we argue by induction), since the problem and its data are invariant
by orthogonal transformations in the first (respectively, in the last)
$m$ variables $x_i$, the solution $\usup_{R,k+1}$ depends only on $s$
and $t$.

Using that
$\usup_{R,0}(x)=u_0(z)$ is a supersolution of \eqref{exisusup}
(as stated in Proposition~\ref{prop}),
let us show that the sequence $\usup_{R,k}$ is nonincreasing in $k$.
More precisely,
$$
u_0(z)=\usup_{R,0}\geq \usup_{R,1}\geq \dots \geq
\usup_{R,k}\geq \usup_{R,k+1}\geq \dots\geq 0 \quad \text{in } T_R.
$$
Indeed, we have
$$\{ -\Delta  -f'(M) \} \usup_{R,1}= f(\usup_{R,0})-f'(M)\usup_{R,0} \leq \{ -\Delta -f'(M) \} \usup_{R,0}$$
in $T_R$. Hence $\usup_{R,1}\leq\usup_{R,0}$.
We use now that $0$ a subsolution of \eqref{exisusup}.
Since $\{ -\Delta-f'(M) \}0=0\leq
g(\usup_{R,0})$ in $T_R$ and $0\leq u_0(z)$ on
$\partial T_R$, we deduce
$\usup_{R,1}\geq 0$.

Assume now that $0\leq \usup_{R,k}\leq \usup_{R,k-1}$ in $T_R$, for some $k\geq 1$.
Then, since $g$ is increasing in $(0,M)$ (see Remark~\ref{f'dec}),
$g(\usup_{R,k})\leq \nolinebreak g(\usup_{R,k-1})$, and thus
$$\{ -\Delta  -f'(M) \} \usup_{R,k+1} = g(\usup_{R,k}) \leq g(\usup_{R,k-1})=\{ -\Delta -f'(M) \} \usup_{R,k}.$$
Again by the maximum principle $\usup_{R,k+1}\leq
\usup_{R,k}$. Besides, $\usup_{R,k+1}\geq 0$ since $g(\usup_{R,k})\geq 0$. 

Next, by monotone convergence this sequence converges to a nonnegative
solution $\usup_R$ of $-\Delta u=f(u)$ in $T_R$, 
which depends only on $s$ and $t$, and
such that $\usup_R\equiv 0$ on ${\mathcal C}\cap\overline{T_R}$. Since
$\usup_R=u_0(z)$ on $\{s=R\}\cap \overline{T_R}$, $\usup_R$ is not identically
$0$, and thus the strong maximum
principle and $f(0)=0$ lead to $\usup_R >0$ in $T_R$. 

Moreover, $\usup_R$
is maximal as stated in the lemma. Indeed, by assumption $0<u\leq u_0(z)$ in
all $T_R$. Assume now that $0<u\leq \usup_{R,k}$ for some $k\geq 0$. 
We then have
$$
\{ -\Delta  -f'(M) \} u = g(u) \leq g(\usup_{R,k})=\{ -\Delta -f'(M) \}
\usup_{R,k+1}
$$ 
in $T_R$ and $u\leq u_0(z)= \usup_{R,k+1}$ on $\partial T_R$.
Thus $u\leq \usup_{R,k+1}$ in $T_R$. By induction $u\leq \usup_{R,k+1}$ 
for all $k$ and hence
$u\leq \usup_R$ in $T_R$.

Finally, the semi-stability of $\usup_R$ in $T_R$ follows directly from
Remark~\ref{stability}.
\end{proof}

\begin{proposition}\label{max}
Let $f$ satisfy conditions  \eqref{H}. Then,
there exists a positive solution $\usup$ of 
\begin{equation}\label{dirmax}
\left\{
\begin{array}{rcll}
-\Delta \usup  & = & f(\usup) &
  \text{ in } {\mathcal O}=\{s>t\}\\
\usup & = & 0 &  \text{ on } {\mathcal C}= \partial {\mathcal O},
\end{array}
\right.
\end{equation}
with $0<\usup <M$ in ${\mathcal O}$, which is maximal in the following sense.
We have that $\usup\geq u$ in ${\mathcal O}$
for every bounded solution $u$ of $-\Delta u=f(u)$ 
in $\R^{2m}$ that vanishes on the Simons cone and
has the same sign as $s-t$.

In addition, $\usup$ depends only on $s$
and $t$.
\end{proposition}

\begin{proof}
By elliptic estimates and a compactness argument, 
the limit as $R\to\infty$ of the solutions $\usup_{R}$ of Lemma~\ref{maxR} 
exists (up to subsequences) in every compact set of $\overline {\mathcal O}$. 
We obtain a solution $\usup$ in
${\mathcal O}=\{s>t\}$ such that $\usup=0$ on ${\mathcal C}$
and $0\leq \usup \leq M$ in ${\mathcal O}$.
Clearly, $u$ depends only on $s$ and $t$.

Let us next establish the maximality of $\usup$. Let $u$ be a solution
as in the statement. By Proposition~\ref{prop}, 
we have $u\leq u_0(z)$ in ${\mathcal O}$. Thus, by Lemma~\ref{maxR},
$u\leq \usup_{R}$ in $T_R$ for all $R$. 
It follows that $u\leq \usup$ in ${\mathcal O}$.

Finally, we show that $\usup$ is not identically $0$. Indeed,
by maximality of $\usup$ and the existence of saddle
solution $u$ of Theorem~\ref{exis}, we deduce that $\usup\geq u >0$ in 
${\mathcal O}$. An alternative way to prove $u\not\equiv 0$ is to use the
subsolution of next remark, by placing such subsolution below
all maximal solutions $\usup_{R}$. 
\end{proof}

\begin{remark}
\label{subsol}
{\rm
The following subsolution is useful in several arguments.
Let $R$ be large enough such that ${\mathcal O}_R={\mathcal O}\cap B_R$ 
contains a closed ball $B$ of sufficiently large radius to guarantee that 
$\lambda_1<f'(0)$, where $\lambda_1$ is the first Dirichlet eigenvalue of the
Laplacian in $B$. 
Let $\phi_1>0$ in $B$ be the first Dirichlet eigenfunction of $-\Delta$ in $B$ and
$\varepsilon>0$ a constant small enough. Then,
$$-\Delta(\varepsilon\phi_1)=\lambda_1 \varepsilon\phi_1\leq
f(\varepsilon\phi_1)\; \text{ in } B.$$ 
The last inequality holds since 
$f(\varepsilon\phi_1)/(\varepsilon\phi_1)\geq f'(\varepsilon\phi_1)$ for $\varepsilon$
small (recall that $f$ is concave in $(0,M)$),
and also $f'(\varepsilon\phi_1)>\lambda_1$ for $\varepsilon$ small, since we chose
$B$ such that $f'(0)>\lambda_1$. Therefore, 
$\varepsilon\phi_1$ extended by zero in ${\mathcal O}_R\setminus B$ is a
subsolution to problem
\begin{equation*}
\left\{\begin{array}{rcll}
-\Delta v& = & f(v) & \text{ in } {\mathcal O}_R\\
v& = & 0 & \text{ on } \partial{\mathcal O}_R,
\end{array}\right.
\end{equation*}
which is positive in $B$, a ball with compact closure contained in
${\mathcal O}_R$.
}
\end{remark}

Next we start studying the existence of a minimal saddle solution.
Among other things, the next lemma establishes 
the uniqueness
of positive solution in ${\mathcal O}_R$ taking values in $(0,M)$ and vanishing
on $\partial{\mathcal O}_R$. This uniqueness is a well known general
result that only requires $f(\rho)/\rho$ to be decreasing in
$(0,M)$.

\begin{lemma}\label{minR}
Let $f$ satisfy conditions \eqref{H}. Then, for $R$ large enough, there
exists a unique
positive solution $\usub_R$ of 
\begin{equation}\label{dirunique}
\left\{
\begin{array}{rcll}
-\Delta u & = & f(u) &
  \text{ in } {\mathcal O}_R:=\{ s>t \}\cap B_R\\
u & < & M &  \text{ in }\; {\mathcal O}_R\\
u & = & 0 &  \text{ on }\; \partial {\mathcal O}_R,
\end{array}
\right.
\end{equation}
which is minimal in the following sense. We have that $\usub_R\leq u$ in
 ${\mathcal O}_R$ 
for every positive solution $u$ of $-\Delta u=f(u)$ in ${\mathcal O}_R$.

Moreover, $\usub_R$ is semi-stable in ${\mathcal O}_R$ in the sense of 
Definition~\ref{defstablebounded}, and $\usub_R$ depends only on $s$ and $t$.
\end{lemma}

\begin{proof}
We claim
that, for $R$ large enough, there exists 
a maximal positive solution $\overline{v}$ of \eqref{dirunique}.
This is proved by monotone iteration as in the proof of Lemma
\ref{maxR}, now with zero boundary conditions on $\partial{\mathcal O}_R$, 
and starting the iteration with $\overline{v}_{R,0}\equiv M$, a
supersolution of \eqref{dirunique}.
Here, we use the subsolution of Remark~\ref{subsol} to guarantee
that the limit of the iteration, $\overline{v}$, is not identically
zero, and thus positive.

Using this maximal solution, we can now prove the uniqueness statement
of the lemma. 
Let $v$ be a positive solution of \eqref{dirunique}.
Since $\overline{v}$ is maximal
we have $v\leq \overline{v}$ in ${\mathcal O}_R$.

To avoid integrating by parts in all of ${\mathcal O}_R$ (since
$\partial{\mathcal O}_R$ is not a Lipschitz domain at 
$0\in\partial{\mathcal O}_R$), we consider
a ball $B_{\varepsilon}$ centered at the origin and of radius
$\varepsilon$. The integration by parts formula can now be used in the
region ${\mathcal O}_R\setminus B_{\varepsilon}$.
Multiplying the equation for $\overline{v}$ by ${v}$ and integrating over
${\mathcal O}_R$  we get
\begin{eqnarray*}
\int_{{\mathcal O}_R} f(\overline{v}){v} &  = & 
\hspace{-1em}\int_{{\mathcal O}_R\cap B_{\varepsilon}} (-\Delta
\overline{v}){v}+  \int_{{\mathcal O}_R\setminus B_{\varepsilon}} (-\Delta
\overline{v}){v} \\
& = &\hspace{-1em}\int_{{\mathcal O}_R\cap B_{\varepsilon}} (-\Delta
\overline{v}){v}+\int_{{\mathcal O}_R\setminus B_{\varepsilon}} \nabla
\overline{v}\nabla {v} -\int_{\partial B_{\varepsilon}\cap {\mathcal O}_R}
{v} \nabla \overline{v}\cdot \nu ,
\end{eqnarray*}
where $\nu$ is the outward normal on $\partial B_{\varepsilon}$ to
 ${\mathcal O}_R\setminus B_{\varepsilon}$.
 Similarly, we multiply now the equation for ${v}$ by $\overline{v}$ and obtain
\begin{eqnarray*}
\int_{{\mathcal O}_R} f({v})\overline{v} &  = & 
\hspace{-1em}\int_{{\mathcal O}_R\cap B_{\varepsilon}} (-\Delta
{v})\overline{v}+  \int_{{\mathcal O}_R\setminus B_{\varepsilon}} (-\Delta
{v}) \overline{v}\\
& = &\hspace{-1em} \int_{{\mathcal O}_R\cap B_{\varepsilon}} (-\Delta
{v})\overline{v}+\int_{{\mathcal O}_R\setminus B_{\varepsilon}} \nabla
{v}\nabla \overline{v} -\int_{\partial B_{\varepsilon}\cap {\mathcal O}_R}
\overline{v}\nabla{v}\cdot \nu .
\end{eqnarray*}
Subtracting, we obtain
\[
\int_{{\mathcal O}_R} \left(\frac{f(\overline{v})}{\overline{v}}-
\frac{f({v})}{{v}}\right)\overline{v}{v}=\int_{{\mathcal O}_R\cap
  B_{\varepsilon}} \Big( (\Delta
{v})\overline{v}-(\Delta \overline{v}){v}\Big)-
\]
\[-\int_{\partial B_{\varepsilon}\cap {\mathcal O}_R}
\Big({v}\nabla \overline{v}\cdot\nu-\overline{v}\nabla{v}\cdot \nu \Big)  .
\]

We claim that every positive solution $v$ of \eqref{dirunique} belongs to
$C^2(\overline{{\mathcal O}_R})$, that is, is $C^2$ up to the
boundary. This is proved using the ideas of the proof in \cite{CT} of Theorem
\ref{exis} above. Namely, we first do odd reflection of ${v}$ with respect to
${\mathcal C}$ to obtain a solution in $B_R\setminus\{0\}$. Then, by a
standard capacity argument we see that ${v}$ is a solution in fact in all
$B_R$. Then, classical elliptic theory gives ${v}\in
C^{2,\alpha}(\overline{B_R})$. 

Thus, we can let $\varepsilon$ tend to zero above and obtain
\begin{equation}\label{uni}
\int_{{\mathcal O}_R} \left(\frac{f(\overline{v})}{\overline{v}}-
\frac{f({v})}{{v}}\right)\overline{v}{v}=0.
\end{equation}
Since $0<{v}\leq \overline{v}<M$, Remark \ref{f'dec} leads to
$$\frac{f(\overline{v})}{\overline{v}} - \frac{f(v)}{v}\leq 
0 \quad\text{ in } {\mathcal O}_R.$$
Thus, the integrand in \eqref{uni} is nonpositive. Its integral
being zero leads to such integrand being identically zero, and thus
$v\equiv \overline{v}$.

We have proved that, for $R$ large enough, there exists a unique
positive solution of \eqref{dirunique}, that we denote by
$\usub_R$ as in the statement of the lemma.
This solution agrees with the maximal solution $\overline{v}$ obtained
by monotone iteration at the beginning of the proof.
Thus, the solution depends only on $s$ and $t$. In addition, its semi-stability
in ${\mathcal O}_R$ follows from Remark~\ref{stability}.

It only remains to prove the statement on minimality. For this, let $u$
be positive and satisfy $-\Delta u =f(u)$ in ${\mathcal O}_R$.
Choose $\varepsilon >0$ small enough such that the subsolution
of Remark~\ref{subsol} is smaller than $\min(u,M)$ (a supersolution) 
in all of ${\mathcal O}_R$.
Then, in between this subsolution and supersolution there is a solution of
\eqref{dirunique} which, by uniqueness, must coincide with
$\usub_R$. Thus $\usub_R \leq \min(u,M)\leq u$ in ${\mathcal O}_R$ as claimed.
\end{proof}

\begin{proposition}\label{min}
Let $f$ satisfy conditions \eqref{H}. Then, there exists a 
positive solution $\usub$ of 
\begin{equation}\label{dirmin}
\left\{
\begin{array}{rcll}
-\Delta \usub  & = & f(\usub) &
  \text{ in } {\mathcal O}\\
\usub & = & 0 &  \text{ on }\; {\mathcal C}= \partial {\mathcal O},
\end{array}
\right.
\end{equation}
which is minimal in the following sense.
We have that $\usub\leq u$ in ${\mathcal O}$
for every bounded solution $u$ of $-\Delta u=f(u)$ 
in $\R^{2m}$ that vanishes on the Simons cone and
has the same sign as $s-t$.

In addition, $\usub$ depends only on $s$
and $t$.

\end{proposition}

\begin{proof}
Note that, by minimality, the solutions $\usub_R$ of the previous lemma
form an increasing sequence in
$R$. Thus, letting $R\to\infty$ we obtain a positive solution $\usub$,
which depends only on $s$
and $t$. Finally, the statement on minimality of $\usub$ follows
immediately from the one for $\usub_R$ in Lemma~\ref{minR}.
\end{proof}

\subsection{Monotonicity properties}\label{susct342}

We now prove monotonicity properties for the maximal
solution $\usup$ of Proposition~\ref{max}.

\begin{lemma}\label{maxRmont}
The maximal solution $\usup_R$ of Lemma \ref{maxR} satisfies  
$\partial_t \usup_R \leq 0$ in~$T_R$.
\end{lemma}

\begin{proof}
We will see that we can prove the monotonicity property with two different methods,
one based on the semi-stability of $\usup_R$, and the other
on the maximality of  $\usup_R$ and thus the monotone iteration
method to construct it.
In any case, we first must check the right monotonicity on 
the boundary $\partial T_R$.

Since $\usup_R\geq 0$ in $T_R$ and vanishes on $\{s=t\leq R\}$, we have that
$\partial_t\usup_R\leq 0$ on $\{s=t\leq R\}$. On the remaining part of 
$\partial T_R$, which is $\{t<s=R\}$, $\usup_R(x)=u_0(z)=u_0((s-t)/\sqrt{2})$. Thus,
 for $t<s=R$,
$$
\partial_t\usup_R(R,t)=-\frac{1}{\sqrt{2}}\dot{u}_0(z)<0.
$$
Therefore $\partial_t\usup_R\leq 0$ on $\partial T_R$.

It follows that the positive part $(\partial_t\usup_R)^+$ belongs to 
 $H^1_0 (T_R)$, and hence
it is an admissible test function for the quadratic form
$Q_{\usup_R}$ in $T_R$.
Recall that by Lemma \ref{maxR}, $\usup_R$ is semi-stable in $T_R$, that is,
$$
Q_{\usup_R}(\xi)=\int_{T_R}\left\{ |\nabla\xi|^2-f'(\usup_R)\xi^2\right\} dx
\geq 0,
$$ 
for all $\xi\in H_0^1(T_R)$.
Now, since $\usup_R$ is a solution, in coordinates
$s$ and $t$ we have
$$
-\partial_{ss}\usup_R-\partial_{tt}\usup_R- (m-1)\frac{\partial_s\usup_R}{s}
- (m-1)\frac{\partial_t\usup_R}{t}=f(\usup_R) \; \text{ in } T_R \cap \{t>0\}.
$$
Differentiating with respect to $t$ we get
\begin{equation}\label{349}
-\Delta \partial_t\usup_R -f'(\usup_R)\partial_t\usup_R =
-\frac{m-1}{t^2}\partial_t\usup_R\quad \text{ in } T_R\cap \{t>0\}.
\end{equation}
Therefore, multiplying \eqref{349} by $(\partial_t\usup_R)^+$
and integrating by parts, we have that
\begin{eqnarray*}
0 &\leq & Q_{\usup_R}((\partial_t\usup_R)^+)\\
& = & \int_{T_R}\left\{ |\nabla(\partial_t\usup_R)^+|^2
-f'(\usup_R)((\partial_t\usup_R)^+)^2\right\} dx\\
 & = & -\int_{T_R}\frac{m-1}{t^2}((\partial_t\usup_R)^+)^2dx \leq 0;
\end{eqnarray*}
note that the set $\{t=0\}$ is of zero measure.
Since the last integrand is nonnegative, its integral being zero leads to 
$(\partial_t\usup_R)^+\equiv 0$ in $T_R\cap\{t>0\}$ and thus in $T_R$.
This finishes the proof.

As mentioned in the beginning of the proof, $\partial_t\usup_R\leq 0$
can also be established using the maximality of the solution. Indeed,
by its maximality, $\usup_R$ must be equal to the solution constructed
by monotone iteration in \eqref{exisusup}. Assuming 
$\partial_t\usup_{R,k}\leq 0$ (which clearly holds for $k=0$),
we differentiate the equation in \eqref{exisusup} to obtain
$$
\left\{ -\Delta -f'(M) +\frac{m-1}{t} \right\}\partial_t \usup_{R,k+1} 
= g'(\usup_{R,k}) \partial_t \usup_{R,k}
\quad  \text{ in } T_R.
$$
The right hand side is nonpositive by inductive hypothesis.
Since the operator $-\Delta -f'(M) +(m-1)/t$ satisfies the
maximum principle (due to the positive signs of the zeroth order
coefficients), we deduce $\partial_t\usup_{R,k+1}\leq 0$.
\end{proof}

In a similar way we now establish a sign for $\partial_y\usup_R$
in $T_R$.

\begin{lemma}\label{maxmony}
The maximal solution $\usup_R$ of Lemma \ref{maxR} satisfies  
$\partial_y \usup_R \geq 0$ in $T_R$.
\end{lemma}

\begin{proof}
We first check that
$\partial_y\usup_R\geq 0$ on $\partial T_R$. To see this, simply note that
$\usup_R\equiv 0$ on the part of the boundary in the Simons cone, $\{t=s<R\}$. Since
$\partial_y$ is a tangential derivative here, we have
$\partial_y\usup_R= 0$ in $\{t=s<R\}$. Take now a point $(s=R,t)$ with $0<t<R$
on the remaining part of
the boundary. Recall that $\usup_R\leq u_0(z)$ in all $T_R$. 
Thus, for all $0<\delta < t$, we have $\usup_R (R-\delta, t-\delta) \leq 
u_0((R-\delta-(t-\delta))/\sqrt{2})=u_0((R-t)/\sqrt{2})=\usup_R (R, t)$.
Differentiating with respect to $\delta$ at $\delta=0$, we deduce
$\partial_y\usup_R (R,t)\geq 0$.

Thus $\partial_y\usup_R\geq 0$ on $\partial T_R$
and hence we can take
$(\partial_y\usup_R)^-$ as a test function in $Q_{\usup_R}$.

As in \eqref{349}, we now use the analogue equation for $\partial_s\usup_R$:
\begin{equation*}
-\Delta \partial_s\usup_R -f'(\usup_R)\partial_s\usup_R =
-\frac{m-1}{s^2}\partial_s\usup_R\quad \text{ in } T_R.
\end{equation*}
From this and \eqref{349}, since $\partial_y=(\partial_s+\partial_t)/\sqrt{2}$, we obtain
\begin{eqnarray*}
-\Delta\partial_y\usup_R-f'(\usup_R)\partial_y\usup_R & =
 & -\frac{m-1}{\sqrt{2}}\left(\frac{\partial_s\usup_R}{s^2}+
\frac{\partial_t\usup_R}{t^2}\right)\\
&  & \hspace{-18mm} =\, -\frac{m-1}{s^2}\partial_y\usup_R - 
\frac{(m-1)(s^2-t^2)}{\sqrt{2}s^2t^2}\partial_t\usup_R
\end{eqnarray*}
in $T_R\cap\{t>0\}$. Thus,
\begin{equation}\label{350}
-\Delta (-\partial_y\usup_R)-f'(\usup_R)(-\partial_y\usup_R)  =
-\frac{m-1}{s^2}(-\partial_y\usup_R) + 
\frac{(m-1)(s^2-t^2)}{\sqrt{2}s^2t^2}\partial_t\usup_R
\end{equation}
in $T_R\cap\{t>0\}$.

Multiplying \eqref{350} by $(\partial_y\usup_R)^-$
and integrating by parts, we have that
\begin{eqnarray*}
0 & \leq & Q_{\usup_R}((\partial_y\usup_R)^-)\\
&=& \int_{T_R}\left\{ |\nabla(\partial_y\usup_R)^-|^2
-f'(\usup_R)((\partial_t\usup_R)^-)^2\right\} dx\\
 & = & -\int_{T_R}\frac{m-1}{s^2}((\partial_y\usup_R)^-)^2dx
+ \\
&  & \hspace{4em} 
+\int_{T_R}\frac{(m-1)(s^2-t^2)}{\sqrt{2}s^2t^2}
\partial_t\usup_R (\partial_y\usup_R)^- dx \leq 0,
\end{eqnarray*}
since $\partial_t\usup_R\leq 0$ in $T_R$ by Lemma \ref{maxRmont}.
Therefore $(\partial_y\usup_R)^-\equiv 0$ in $T_R$, which
finishes the proof.
\end{proof}

Finally we prove Theorem \ref{teominmax}.

\begin{proof}[Proof of Theorem \ref{teominmax}]
By Propositions \ref{min} and \ref{max} we know that there exist a
positive minimal solution $\usub$ and a positive maximal solution $\usup$ in 
${\mathcal  O}$ in the sense stated in the propositions.
In addition, they depend only on $s$ and $t$. 

Now, since $f$ is odd, by odd reflection with respect to ${\mathcal C}$ we
obtain saddle solutions $\usub$ and $\usup$ in
$\R^{2m}$ such that
$$
|\usub|\leq |u|\leq|\usup| \quad {\rm in}\, \R^{2m},
$$ 
for every solution $u$ 
of $-\Delta u=f(u)$ in $\R^{2m}$ that vanishes on the Simons cone
and has the same sign as $s-t$.

It only remains to prove the monotonicity properties stated in the theorem.
First, since $\usup$ is the limit of $\usup_R$ as $R\to\infty$,
Lemmas~\ref{maxRmont} and \ref{maxmony} give 
that $-\partial_t \usup\geq 0$ and $\partial_y \usup\geq 0$ in $\{s>t\}$.
As a consequence, $\partial_s \usup\geq 0$ in $\{s>t\}$.

It follows, since $u(s,t)=-u(t,s)$, that $-\partial_t \usup\geq 0$
in all of $\R^{2m}$. Our equation reads
\begin{equation}\label{351}
-\partial_{ss}\usup-\partial_{tt}\usup- (m-1)\frac{\partial_s\usup}{s}
- (m-1)\frac{\partial_t\usup}{t}=f(\usup) \; \text{ in } T_R \cap \{t>0\}
\end{equation}
and, differentiating it with respect to $t$, we also have
\begin{equation}\label{352}
-\Delta \partial_t\usup -f'(\usup)\partial_t\usup =
-\frac{m-1}{t^2}\partial_t\usup\quad \text{ in } T_R\cap \{t>0\}.
\end{equation}
Since $-\partial_t \usup\geq 0$ in $\R^{2m}$, \eqref{352} and the
strong maximum principle give that $-\partial_t \usup > 0$ in $\R^{2m}
\setminus\{t=0\}$. On the other hand, multiplying \eqref{351}
by $t$, using that every saddle solution is of class $C^2$,
and letting $t\to 0$, we deduce $-\partial_t \usup =0$ in $\{t=0\}$.
Statement (a) of the theorem is now proved.

Part (b) is proved in the same way ---or it is simply the symmetric result to (a).
Statement (c) follows directly from (a) and (b).

Finally, we prove (d). 
Equation \eqref{350}, after letting $R\to\infty$, gives
\begin{eqnarray*}
-\Delta\partial_y\usup-f'(\usup)\partial_y\usup & =
 & -\frac{m-1}{s^2}\partial_y\usup - 
\frac{(m-1)(s^2-t^2)}{\sqrt{2}s^2t^2}\partial_t\usup\\
& \geq &  -\frac{m-1}{s^2}\partial_y\usup
\end{eqnarray*}
in $\{s>t>0\}$, since $ \partial_t\usup\leq 0$ in this set.
We have already proved (using Lemma~\ref{maxmony}) that 
$\partial_y \usup\geq 0$ in $\{s>t\}$. Hence, the strong maximum principle
leads to $\partial_y \usup > 0$ in $\{s>t\}$, as claimed.
\end{proof}

Finally we prove Proposition \ref{cstable}, which states that every
solution that vanishes on the Simons
cone ${\mathcal C}$ and has the same sign as $s-t$ is stable in ${\mathcal O}$.

\begin{proof}[Proof of Proposition \ref{cstable}]
Let $u$ be a bounded solution  
of $-\Delta u=f(u)$ in $\R^{2m}$ that vanishes on the Simons cone
${\mathcal C}=\{s=t\}$ and has the same sign as $s-t$.
Then, by Proposition~\ref{prop}, we have that $0<u<M$ in ${\mathcal O}=\{ s>t\}$. 
Thus, by Remark~\ref{stability}, 
$Q_u(\xi) \geq 0$ for all
$\xi\in H^1(\R^{2m})$ with compact support and such that $\xi \equiv 0$ in
$\R^{2m}\setminus \{ s>t\}$, where $Q_u$ is defined in \eqref{stable}.

By the analogue (or symmetric) argument now in $\{ s<t\}$ (instead of
$\{ s>t\}$), and since $-M<u<0$ in $\{ s<t\}$ by Proposition~\ref{prop},
$Q_u(\xi) \geq 0$ for all
$\xi\in H^1(\R^{2m})$ with compact support and such that $\xi \equiv 0$ in
$\R^{2m}\setminus \{ s<t\}$.

Now, given
$\xi\in C^1(\R^{2m})$ with compact support and with $\xi \equiv 0$ in
${\mathcal C}$, we write $\xi= \chi_{\{s>t\}}\xi + \chi_{\{s<t\}}\xi$,
the sum of two $H^1$ functions, and we use the previous facts to conclude
$Q_u(\xi) \geq 0$.
\end{proof}

\section{Asymptotic behavior of saddle solutions in $\R^{2m}$}\label{sctasym}

This section is devoted to study the asymptotic behavior at infinity of
saddle-shaped solutions to
$-\Delta u=f(u)$ in $\R^{2m}$, and more generally, of solutions
(not necessarily depending only on $s$ and $t$) which are
odd with respect to the Simons cone ${\mathcal C}$ and positive in 
${\mathcal O}=\{s>t\}$.

We will work in the $(y,z)$ system of coordinates. Recall that we
defined, in ($\ref{defyz}$), $y$ and $z$ by
\begin{center}
$
\left\{ \begin{array}{rcl}
y&=&{\ds (s+t)/\sqrt{2}}\\
z&=&{\ds (s-t)/\sqrt{2}},\\
\end{array}
\right.
$
\end{center}
which satisfy $y\geq 0$ and $-y\leq z \leq y$.

We prove Theorem~\ref{asym}, which states that any solution $u$ as above 
tends at infinity to the function
$$
U(x):=u_0(z),
$$
uniformly outside of compact sets. Similarly the gradient of $u$,
$\nabla u$, converges to $\nabla U$.  This fact will be
important in the proof of instability of saddle solutions in dimension~$6$.

Our proof of Theorem~\ref{asym}, which argues by contradiction, uses a well 
known compactness argument based on translations of the solution, as well as 
two crucial classification or Liouville type results for monostable equations.
Regarding the nonlinearity, both assume that 
$g:[0,+\infty)\rightarrow\R$ is a $C^1$ function such that
\begin{equation}\label{condg}
g(0)=g(1)=0,\ g'(0)>0,\ g>0 \text{ in } (0,1),\ \text{ and }\ g<0 \text{ in } 
(1,+\infty). 
\end{equation}

The first result concerns global solutions, that is solutions in all space,
and is originally due to Aronson and Weinberger~\cite{AW}; the statement that
we present, a simpler proof, and much more general
results are due to Berestycki, Hamel, and Nadirashvili, see 
Proposition 1.14 of \cite{BHN} (see also \cite{BHR} for more general results).

\begin{proposition} $(${\bf Aronson-Weinberger} \cite{AW};
{\bf Berestycki-Hamel-Nad-\-irashvili \cite{BHN}$)$.} \label{BHN}
Let $g$ satisfy \eqref{condg} 
and let $b\in\R^n$ be such that $|b|<2\sqrt{g'(0)}$. Let $u$ be a
bounded solution of
\begin{equation*}
\left\{ \begin{array}{rl}
        \Delta u+b\cdot\nabla u+g(u)=0 & \textrm{ in } \R^n\\
                               u\geq 0 & \textrm{ in } \R^n.
         \end{array}\right.
\end{equation*}
Then, $u\equiv 0$ or $u\equiv 1$.
\end{proposition}

The second result concerns the same equation but in a half-space.
It states the 1D symmetry of its solutions and is due to Angenent
\cite{An}. For Lipschitz nonlinearities and slightly weaker
assumptions on them, see Berestycki-Caffarelli-Nirenberg~\cite{BCNm}.

\begin{proposition}[{\bf Angenent} \cite{An}]\label{An}
Let $g$ satisfy \eqref{condg} and also $g'(1)<0$. 
Let $u$ be a bounded solution of
\begin{equation*}
\left\{ \begin{array}{rl}
        \Delta u+g(u)=0 & \textrm{ in } \R^n_+=\{x_n>0\}\\
                               u=0 & \textrm{ on } \partial\R^n_+\\
                               u > 0 & \textrm{ in } \R^n_+.
         \end{array}\right.
\end{equation*}
Then, $u$ depends only on the $x_n$ variable. In addition,
such solution depending only on $x_n$ exists and is unique.
\end{proposition}

The two previous propositions use Serrin's sweeping method
as important tool in their proofs (see \cite{An} and 
Remark 4.3 of \cite{BHN}).

A very
simple lemma that we will use in the proof of Theorem~\ref{asym}
is the following
(see Lemma~4.2 of \cite{CT} and its proof there).

\begin{lemma}[\cite{CT}]\label{lemmadist}
For every $x=(x^1, x^2)\in\R^{2m}$, 
the distance from $x$ to
the Simons cone ${\mathcal C}$ is given by
${\rm dist}(x,{\mathcal C})=|s-t|/\sqrt{2}$, where 
$s=|x^1|$ and
$t=|x^2|$. In addition, if $s=|x^1|\neq 0$ and
$t=|x^2|\neq 0$, then
${\rm dist}(x,{\mathcal C})=|x-x_0|$ where
$x_0\in{\mathcal C}$ has coordinates
$$
x_0=(\alpha x^1, \beta x^2)
$$ 
and $\alpha$ and $\beta$ are defined by
$\alpha s=\beta t=(s+t)/2$.
\end{lemma}

\begin{proof}[Proof of Theorem \ref{asym}]
Suppose that the theorem does not hold. Thus, let $u$ be a solution of 
$-\Delta u=f(u)$ in $\R^{2m}$
that vanishes on the cone ${\mathcal C}$, is positive in the region
${\mathcal O}=\{s>t\}$ and odd with respect to ${\mathcal C}$, and for which
there exists
$\varepsilon>0$ and a sequence
$\{x_k\}$ with 
\begin{equation}\label{contra}
|x_k|\rightarrow\infty\quad \text{  and } \quad|u(x_k)-U(x_k)|+|\nabla u(x_k)-\nabla U(x_k)|  
> \varepsilon.
\end{equation}
By continuity  we may move slightly $x_k$ and assume
$x_k\not\in{\mathcal C}$ for all $k$. Moreover, up to a
subsequence (which we still denote by $\{x_k\}$), either
$\{x_k\}\subset\{s>t\}$ or  $\{x_k\}\subset\{s<t\}$. By the symmetries
of the problem we may
assume  $\{x_k\}\subset\{s>t\}={\mathcal O}$.

Recall that, by Proposition \ref{prop}, we know that $|u(x)|\leq
|u_0(z)|<M$ in $\R^{2m}$. We  distinguish  two cases:

\vspace{1em}

{\sc Case 1.} $\{{\rm dist}(x_k,{\mathcal C})=:d_k\}$ is an unbounded sequence.

In this case, since $0<z_k={\rm dist}(x_k,{\mathcal C})=d_k\rightarrow
+\infty$ (for a subsequence), we have that $U(x_k)=u_0(z_k)=u_0(d_k)$
tends to $M$
and $|\nabla U(x_k)|$ tends to $0$, that is,
$$U(x_k) \rightarrow M\quad {\rm and}\quad |\nabla
U(x_k)|\rightarrow 0.$$
From this and \eqref{contra} we have
\begin{equation}\label{case1}
|u(x_k)-M| + |\nabla u(x_k)|\geq\frac{\varepsilon}{2},
\end{equation}
for $k$ large enough.
 
Consider the ball $B_{d_k}(0)$ of radius $d_k$ centered at $0$, and define
$$v_k(\tilde{x})=u(\tilde{x}+x_k),\qquad\text{for }\tilde{x}\in B_{d_k}(0).$$
Since $B_{d_k}(0)+x_k\subset\{s>t\}$ we have that $0<v_k<M$ in $B_{d_k}(0)$ and
$-\Delta v_k = f(v_k)$ in $B_{d_k}(0)$. Letting $k$ tend to infinity we
obtain, through a subsequence, 
a solution $v$ to the problem in all of $\R^{2m}$. That is, $v$
satisfies
\begin{equation}\label{solpos}
\left\{
\begin{array}{rl}
 -\Delta v=f(v) & {\rm in}\; \R^{2m} \\
  v\geq 0 & {\rm in}\; \R^{2m}.
\end{array}
\right.
\end{equation}

Define $w=v/M$ where $M$ is given by the assumptions on $f$. Then $w$
is a nonnegative solution of 
$-\Delta w=g(w)$ in $\R^{2m}$,
where $g(w)=\frac{1}{M}f(Mw)$. Since $w\leq 1$ in $\R^{2m}$, we may
change $g$ in $(1,+\infty)$ to be negative here (for instance, taking
$g$ linear in $(1,+\infty)$ with slope $g'(1)<0$ and $g(1)=0$), and
$w$ is still a solution of the same equation. Since $f$ satisfies \eqref{H}, 
the function $g$ (changed as  above in $(1,+\infty)$) satisfies the conditions of
Proposition~\ref{BHN} with $b\equiv 0$. Hence we have that $w\equiv 0$ or
$w\equiv 1$, that is, $v\equiv 0$ or $v\equiv M$. In either case, 
$\nabla v(0)=0$, that is,
$|\nabla u(x_k)|$ tends to
$0$. 

Next we show that $v\not\equiv 0$. By Proposition \ref{cstable}
(established in the previous section), we have that $u$ is stable in
${\mathcal O}$. 
Hence, $v_k$ is semi-stable in $B_{d_k}(0)$ (since $B_{d_k}(0)+x_k
\subset {\mathcal O}$) in the sense of Definition~\ref{defstablebounded}. 
This implies
that $v$ is
stable in all of $\R^{2m}$ and therefore $v\not\equiv 0$ (otherwise,
since $f'(0)>0$ we could construct a test function $\xi$ such that
$Q_v(\xi)<0$ which would be a contradiction with the fact that $v$ is
stable). 

Hence,
it must be $v \equiv M$. But this implies that $v(0)=M$ and so
$u(x_k)$ tends to $M$.
Hence, we have that $u(x_k)$ tends to $M$ and $|\nabla u(x_k)|$ tends
to $0$, which is a contradiction with ($\ref{case1}$). Therefore, we
have proved the theorem in this case 1.

\vspace{1em}

{\sc Case 2.} $\{{\rm dist}(x_k,{\mathcal C})=:d_k\}$ is a bounded sequence.

The points $x_k$ remain at a finite distance to the cone. Then, at
least for a subsequence, $d_k\rightarrow d\geq 0$ as $k\rightarrow\infty$.
Let $x_k^0\in{\mathcal C}$ be a point that realizes the distance to the cone, that
is,
\begin{equation}\label{case2}
{\rm dist}(x_k,{\mathcal C})=|x_k-x_k^0|=d_k,
\end{equation}
and let $\nu_k^0$ be the inner unit normal to ${\mathcal
  C}=\partial{\mathcal O}$ at $x_k^0$. Note
that $B_{d_k}(x_k)\subset{\mathcal O}\subset\R^{2m}\setminus{\mathcal C}$ and 
$x_k^0\in\partial B_{d_k}(x_k)\cap{\mathcal C}$, i.e., $x_k^0$ is the
point where the sphere $\partial B_{d_k}(x_k)$ is tangent to the cone
${\mathcal C}$. It follows that
$x_k^0\neq 0$ and also that $(x_k-x_k^0)/d_k$ is the unit normal $\nu_k^0$
to ${\mathcal C}$ at $x_k^0$. That is, $x_k=x_k^0+d_k\nu_k^0$. This
will be checked below, in an alternative way, with explicit formulae.
Now, since the sequence $\{\nu_k^0\}$ is bounded, there exists a
subsequence such that
$$\nu_k^0\rightarrow \nu \in\R^{2m}, \quad |\nu|=1.$$

We define 
$$
w_k(\tilde{x})=u(\tilde{x}+x_k^0) \qquad\text{for }\tilde{x}\in\R^{2m}.
$$ 
The functions $w_k$
are all solutions of $-\Delta w=f(w)$ in $\R^{2m}$ and are uniformly
bounded. Hence, by interior
elliptic estimates the sequence $\{w_k\}$ converges locally 
in $C^2$, up to a subsequence, to a
solution in $\R^{2m}$. Therefore we
have that, as $k$ tends to infinity and up to a subsequence,
$$w_k\rightarrow w \quad {\rm and}\quad \nabla w_k\rightarrow \nabla w
\;\text{uniformly on compact sets of } \R^{2m},$$ where $w$ is a solution 
 of $-\Delta w=f(w)$ in $\R^{2m}$. Note that the curvature of ${\mathcal C}$ at
$x_k^0$ goes to zero as $k$ tends to infinity, since ${\mathcal C}$ is
a cone and
$|x_k^0|\rightarrow\infty$ (note that $|x_k^0|\rightarrow\infty$ due to 
$|x_k|\rightarrow\infty$ and $|x_k-x_k^0|=d_k\rightarrow
d<\infty$). Thus, ${\mathcal C}$ at $x_k^0$ is flatter and flatter as 
$k\rightarrow\infty$ and since we translate $x_k^0$ to $0$, the
limiting function $w$ satisfies (all this we will prove below in detail)
\begin{equation}\label{claimH}
\left\{ \begin{array}{rl}
        \Delta w+f(w)=0 & \textrm{ in } H:=\{\tilde{x}\cdot\nu > 0\}\\
                               w=0 & \textrm{ on } \partial H\\
                               w > 0 & \textrm{ in } H,
         \end{array}\right.
\end{equation}
a problem in a half-space.

Now, by Proposition~\ref{cstable}, $u$ is stable for perturbations 
with compact support in ${\mathcal
O}$, and thus $w_k$ (for $k$ large) and $w$ are stable for perturbations with compact
support in $H$ ---see the computation in \eqref{chalf} below for details. 
Therefore $w$ can not be identically zero. 
By assumption \eqref{H} on $f$, we can apply Proposition~\ref{An} above 
(applied to $w/M$ to have $M=1$), and deduce that $w$ is the unique 1D solution, 
that is, the solution depending only on one variable (the orthogonal direction to $\partial H$). 
Hence,
\[
 w (\tilde x) =
u_0(\tilde x \cdot \nu)\quad 
\text{ for all }\tilde{x}\in H.
\]

{From} the definition of $w_k$, and using that $z_k=d_k=|x_k-x_k^0|$ is a
bounded sequence and that $x_k-x_k^0=d_k\nu_k^0$, we have that 
\begin{eqnarray*}
u(x_k) & = & w_k(x_k -x_k^0) = w(x_k - x_k^0) +{\rm o}(1) = 
u_0 ((x_k - x_k^0) \cdot \nu) +{\rm o}(1)  \\
 & = &  u_0((x_k - x_k ^0)\cdot
\nu_k^0)+{\rm o}(1) = u_0(d_k |\nu_k^0|^2) +{\rm o}(1) \\
& = &  u_0(z_k)+{\rm o}(1)= U(x_k)+{\rm o}(1).
\end{eqnarray*}
The same argument can be done for $\nabla u(x_k)$ and $\nabla
U(x_k)$. We arrive to a contradiction with \eqref{contra}.

Finally, we prove in detail the statements in \eqref{claimH}: 
$w > 0$ in $H$ and $w =0$ on $\partial H$,  as well as the identity 
$x_k-x_k^0=d_k\nu_k^0$ used before. 
 
Let $x_k\in{\mathcal O}$ and $x_k^0\in{\mathcal C}$ as in
\eqref{case2}.
Since $|x_k|\to\infty$
and $\textrm{dist}(x_k,{\mathcal C})=|s_k-t_k|/\sqrt{2}$ remains bounded,
we have $s_k\neq 0$ and $t_k\neq 0$ for $k$ large enough.
Thus, by Lemma~\ref{lemmadist} we may write the coordinates
of $x_k^0$ as  
$$x_k^0=(\alpha x_k^1,\beta x_k^2),\quad {\rm for} \quad 2 \alpha
s_k=2\beta t_k=s_k+t_k.$$
In particular, $|x_k^0|\neq 0$. Then, since we are assuming, without loss of generality, that $x_k\in{\mathcal
  O}=\{s>t\}$ we have that
$$d_k=|x_k-x_k^0|=\frac{s_k-t_k}{\sqrt{2}}.$$
Since ${\mathcal C}=\{s^2-t^2=0\}$ we can write the coordinates of
$\nu_k^0$ as
$$\nu_k^0=\left( \frac{x_k^{0,1}}{|x_k^0|},\dots,\frac{x_k^{0,m}}{|x_k^0|},-
 \frac{x_k^{0,m+1}}{|x_k^0|},\dots,
 -\frac{x_k^{0,2m}}{|x_k^0|}\right).$$ 
Using the previous identities, we conclude that
$$x_k-x_k^0=d_k\nu_k^0\quad {\rm with} \quad
\nu_k^0=(\frac{\alpha x_k^1}{|x_k^0|},-\frac{\beta x_k^2}{|x_k^0|}).$$

Next, for all $\tilde{x}\in\R^{2m}$,
\begin{equation}\label{esc}
 \tilde{x}\cdot
  \nu_k^0=\sum_{i=1}^m\frac{x_k^{0,i}\tilde{x}_i}{|x_k^0|}-\sum_{j=m+1}^{2m}
\frac{x_k^{0,j}\tilde{x}_j}{|x_k^0|} = \frac{x_k^{0,i} \tilde{x}_i - x_k^{0,j} \tilde{x}_j}{|x_k^0|}.
\end{equation}
Since $|\nabla u|\in L^{\infty}(\R^{2m})$ and $u(x_k^0)=0$,
\begin{equation}
\label{vk}
|w_k(\tilde{x})|=|u(x_k^0+\tilde{x})|\leq C {\rm dist}(x_k^0+\tilde{x},{\mathcal C}).
\end{equation}
From Lemma \ref{lemmadist} we have
that ${\rm dist}(x,{\mathcal C})=|s_x-t_x|/\sqrt{2}$.
Recall also that $t_{x_k^0}=s_{x_k^0}$. Now, using \eqref{esc} and sums over
repetead indices $i=1,\ldots ,m$ and $j=m+1,\ldots ,2m$ as before, we have
\begin{eqnarray*}
{\rm dist}(x_k^0+\tilde{x},{\mathcal C}) &\hspace{-1em} = & \hspace{-1em}{\ds \frac{1}{\sqrt{2}}\left|
\sqrt{s_{x_k^0}^2+s_{\tilde{x}}^2+2x_k^{0,i}\tilde{x}_i}-\sqrt{s_{x_k^0}^2+t_{\tilde{x}}^2+
2x_k^{0,j}\tilde{x}_j}\right|}\\
& \hspace{-1em}= & \hspace{-1em}{\ds \frac{1}{\sqrt{2}}\left|
\frac{s_{\tilde{x}}^2-t_{\tilde{x}}^2+2 |x_k^0| \tilde{x} \cdot \nu_k^0}
{\sqrt{s_{x_k^0}^2+s_{\tilde{x}}^2+2x_k^{0,i}\tilde{x}_i}+\sqrt{s_{x_k^0}^2+t_{\tilde{x}}^2+
2x_k^{0,j}\tilde{x}_j}}\right|
}\\
&\hspace{-1em} = & \hspace{-1em}{\ds \left|\frac{1}{\sqrt{2}}
\frac{s_{\tilde{x}}^2-t_{\tilde{x}}^2}{|x_k^0|\left(\sqrt{\frac{1}{2}+\frac{s_{\tilde{x}}^2}
{|x_k^0|^2}
+\frac{2x_k^{0,i}\tilde{x}_i}{|x_k^0|^2}} +\sqrt{\frac{1}{2}+\frac{t_{\tilde{x}}^2}
{|x_k^0|^2}
+\frac{2x_k^{0,j}\tilde{x}_j}{|x_k^0|^2}}\right)}\right.} \\
& & \hspace{-1em}+{\ds \left.\frac{2}{\sqrt{2}}
\frac{ \tilde{x} \cdot \nu_k^0}{\sqrt{\frac{1}{2}+\frac{s_{\tilde{x}}^2}{|x_k^0|^2}
+\frac{2x_k^{0,i}\tilde{x}_i}{|x_k^0|^2}} +\sqrt{\frac{1}{2}+\frac{t_{\tilde{x}}^2}{|x_k^0|^2}
+\frac{2x_k^{0,j}\tilde{x}_j}{|x_k^0|^2}}}\right|}.
\end{eqnarray*}
Now we let $k$ tend to infinity. Since $\tilde{x}$ is fixed,
$|x_k^0|\rightarrow\infty$ and $\nu_k^0\rightarrow\nu$, we obtain
$${\rm dist}(x_k^0+\tilde{x},{\mathcal C})\rightarrow |\tilde{x} \cdot \nu|\quad \text{ as } k\rightarrow\infty.$$
Thus, since
$w_k(\tilde{x})\rightarrow w(\tilde{x})$, from \eqref{vk} we have that $w(\tilde{x})=0$ for
every $\tilde{x}\in\partial H=\{\tilde{x}\cdot\nu=0\}$.

Let now $\tilde{x}\in H$, i.e, $\tilde{x}\cdot\nu>0$. We will prove that $w(\tilde{x})\geq 0$. 
We show that for $k$ large,
$w_k(\tilde{x})=u(x_k^0+\tilde{x})\geq 0$. 
Since $u$ is positive in ${\mathcal O}=\{s>t\}$,
we need to
 establish
that $s_{x_k^0+\tilde{x}}\geq t_{x_k^0+\tilde{x}}$ for $k$ large.
As above, we have that
\begin{eqnarray*}
s_{x_k^0+\tilde{x}}-t_{x_k^0+\tilde{x}} & = & 
\sqrt{s_{x_k^0}^2+s_{\tilde{x}}^2+2x_k^{0,i}\tilde{x}_i}-\sqrt{s_{x_k^0}^2+t_{\tilde{x}}^2
+2x_k^{0,j}\tilde{x}_j}\\
& = & {\ds \frac{s_{\tilde{x}}^2-t_{\tilde{x}}^2+2 |x_k^0| \tilde{x} \cdot \nu_k^0}
{\sqrt{s_{x_k^0}^2+s_{\tilde{x}}^2+2x_k^{0,i}\tilde{x}_i}+\sqrt{s_{x_k^0}^2+t_{\tilde{x}}^2
+2x_k^{0,j}\tilde{x}_j}}},
\end{eqnarray*}
and, letting $k$ tend to infinity, since $\tilde{x}$ is fixed and
$|x_k^0|\rightarrow\infty$, we obtain
\begin{equation}\label{chalf}
s_{x_k^0+\tilde{x}}-t_{x_k^0+\tilde{x}} \rightarrow \sqrt{2}
\tilde{x}\cdot\nu >0 \quad\text{  as }
k\rightarrow\infty.
\end{equation}
Therefore $w(\tilde{x})\geq 0$.
\end{proof}

\section{Instability in dimension $6$}\label{sctuns6}

In this section we prove Theorem \ref{uns6}, establishing that saddle
solutions in dimension $6$ are unstable outside of every compact set. 
The asymptotic analysis done
for $n=4$ in the proof of the instability theorem of \cite{CT} does
not lead to the instability of saddle solutions in dimensions $n\geq 6$.
Indeed,  the supersolution $u_0(z)=u_0((s-t)/\sqrt{2})$ 
is, in every dimension $n=2m\geq 6$,
asymptotically stable at infinity in some weak sense and with respect to 
perturbations $\xi(y,z)$ with separate variables.

Hence, the proof of instability in dimension $6$ requires a more precise argument.
We use the equation satisfied by $\usup_z$, where
$\usup$ is the maximal saddle solution, as well as the monotonicity and
asymptotic properties of  $\usup$. We prove
that $\usup$ is unstable outside of every compact set by constructing 
test functions
$\xi(y,z)=\eta(y)\usup_z(y,z)$ such that $Q_{\usup}(\xi)<0$.
Since $\usup$ is maximal, its instability outside of compact sets 
implies that the same instability 
property holds for all bounded solutions $u$ vanishing on the Simons cone
${\mathcal C}$ and positive in ${\mathcal O}$.

Recall that a  bounded solution $u$  of $-\Delta u=f(u)$ in $\R^{2m}$
is stable provided
$$Q_u(\xi)=\int_{\R^{2m}}\left\{|\nabla\xi|^2-f'(u)\xi^2\right\}dx 
\geq 0\qquad
\text{ for all } \xi\in C_c^{\infty}(\R^{2m}).$$
If $v$ is a function depending only on $s$ and $t$, the quadratic form
$Q_v(\xi)$ acting on perturbations of the form $\xi=\xi(s,t)$ becomes
$$c_m Q_v(\xi)= \int_{\{s> 0, t> 0\}} s^{m-1}t^{m-1}\left\{\xi_s^2+\xi_t^2-
f'(v)\xi^2\right\}dsdt ,$$
where $c_m>0$ is a constant depending only on $m$.
We can further change to variables $(y,z)$ and obtain, 
for a different constant $c_m>0$, 
\begin{equation}\label{five2}
c_m Q_v(\xi)=\int_{\{-y<z<y\}}(y^2-z^2)^{m-1}\left\{\xi_y^2+\xi_z^2-
f'(v)\xi^2\right\}dydz .
\end{equation}

Given the definition of the variables $y$ and $z$,
a function $\xi=\xi(y,z)$ has compact support in $\R^{2m}$ if and only if 
$\xi(y,z)$ vanishes for $y$ large enough.

Now let $u$ be a bounded
solution of $-\Delta u=f(u)$ in $\R^{2m}$ vanishing on the
Simons cone ${\mathcal C}=\{s=t\}$ and positive in ${\mathcal O}=\{s>t\}$.
By Theorem~\ref{teominmax}, we know that 
$$
|u(x)|\leq |\usup(x)|
\qquad\text{ in all of } \R^{2m}.
$$ 
This leads to
$f'(|u(x)|)\geq f'(|\usup(x)|)$ for all $x\in\R^{2m}$, since we
assume $f$ to be concave in $(0,M)$. Now, since $f'$ is even,
we deduce that
$$
f'(u(x))\geq f'(\usup(x))\qquad\text{ for all  } x\in\R^{2m}.
$$ 
Therefore, we conclude
\begin{equation}\label{ineqQ}
Q_u(\xi)\leq Q_{\usup}(\xi)\qquad\text{ for all } 
\xi\in C^{\infty}_c(\R^{2m}).
\end{equation}

It follows that, in order to prove that $u$ is unstable, it
suffices to find a smooth function $\xi$ with compact support 
in $\R^{2m}$ for which $Q_{\usup}(\xi)< 0$. Note also that, by an
approximation argument, it suffices to find a Lipschitz function 
$\xi=\xi(y,z)$, not necessarily smooth,  vanishing for all $y\in (0,+\infty)\setminus I$,
where $I$ is a compact interval in $(0,+\infty)$, and 
for which $Q_{\usup}(\xi)< 0$.

\begin{proof}[Proof of Theorem $\ref{uns6}$] By the previous arguments,
  it suffices to
establish that the maximal solution $\usup$, whose existence is
given by Theorem $\ref{teominmax}$, is unstable outside of every
compact set.

We have, for every test function $\xi$,
$$Q_{\usup}(\xi)=\int_{\R^{2m}}\left(|\nabla
\xi|^2-f'(\usup)\xi^2\right)dx.$$
Suppose now that $\xi=\xi(y,z)=\eta(y, z)\psi(y, z)$,
where $\eta$ and $\psi$ are Lipschitz functions, and $\eta(y,z)=0$ whenever
$y\in (0,+\infty)\setminus I$,
with $I$ a compact interval in $(0,+\infty)$. The expression for $Q_{\usup}$
becomes, since $|\nabla\xi|^2=\xi_y^2+\xi_z^2$,
$$Q_{\usup}(\xi)=\int_{\R^{2m}}\left(|\nabla
\eta|^2\psi^2+\eta^2|\nabla\psi|^2 + 2\eta\psi\nabla\eta\cdot\nabla\psi-f'(\usup)\eta^2\psi^2\right)dx.$$
Integrating by parts and using that 
$
2\eta\psi\nabla\eta\cdot\nabla\psi=\psi\nabla(\eta^2)\cdot\nabla\psi,$
we have
$$Q_{\usup}(\xi)=\int_{\R^{2m}}\left(|\nabla
\eta|^2\psi^2- \eta^2\psi\Delta\psi-f'(\usup)\eta^2\psi^2\right)dx,$$
that is, 
$$Q_{\usup}(\xi)=\int_{\R^{2m}}\left(|\nabla
\eta|^2\psi^2- \eta^2\psi(\Delta\psi+f'(\usup)\psi)\right)dx.$$
Choose now $\psi(y,z)=\usup_z(y,z)$. Since $\usup$ is a
saddle solution of $-\Delta \usup=f(\usup)$ we may differentiate this
equation written in $(y,z)$ variables ---see \eqref{eqyz}--- with respect to $z$ and
find
\begin{equation}\label{linz}
 0=\Delta \usup_z + f'(\usup)\usup_z -\frac{2(m-1)}{y^2-z^2}\usup_z +
\frac{4(m-1)z}{(y^2-z^2)^2}(y\usup_y -z\usup_z).
\end{equation}
Replacing in the expression for $Q_{\usup}$ we obtain,
\[
\hspace{-15em}Q_{\usup}(\xi)=\int_{\R^{2m}}{\Big (}|\nabla\eta|^2{\usup}_z ^2-
\]
\[
\hspace{10em}-\eta^2{\Big
\{}\frac{2(m-1)(y^2+z^2)}{(y^2-z^2)^2}{\usup}_z^2 -
\frac{4(m-1)zy}{(y^2-z^2)^2}{\usup}_y{\usup}_z{\Big \}}{\Big )}dx.
\]
Next we change coordinates to $(y,z)$ and as in \eqref{five2} we have,
for some positive constant $c_m$,
\[
\hspace{-10em}c_mQ_{\usup}(\xi)=\int_{\{-y<z<y\}}(y^2-z^2)^{m-1}{\Big (}|\nabla\eta|^2{\usup}_z ^2-
\]
\[
\hspace{8em}-\eta^2{\Big
\{}\frac{2(m-1)(y^2+z^2)}{(y^2-z^2)^2}{\usup}_z^2 -
\frac{4(m-1)zy}{(y^2-z^2)^2}{\usup}_y{\usup}_z{\Big \}}{\Big )}dydz.
\]

For $a>1$, a constant that we will make tend to infinity, let 
$\eta=\eta(\rho)$ be a Lipschitz function of $\rho:=y/a$ with compact support 
$[\rho_1,\rho_2]\subset (0,+\infty)$. Let us denote by 
$$\eta_a(y)=\eta(y/a)\quad \text{ and }\quad \xi_a(y,z)=\eta_a(y)\usup_z(y,z)=\eta(y/a)\usup_z(y,z),$$ 
the functions named $\eta$ and $\xi$ above. The
change $y=a\rho,dy=ad\rho$ yields

\begin{eqnarray}\label{dim62}
 c_mQ_{\usup}(\xi_a)& = & a^{2m-3}\int_{\{-a\rho <z<a\rho
  \}}\rho^{2(m-1)}(1-
\frac{z^2}{a^2\rho^2})^{m-1}{\Big
(}\eta_{\rho}^2{\usup}_z^2-\\
\nonumber &  & -\eta^2{\Big
\{}\frac{2(m-1)(1+\frac{z^2}{a^2\rho^2})}{\rho^2(1-\frac{z^2}{a^2\rho^2})^2}{\usup}_z^2
- \frac{4(m-1)z}{a\rho^3(1-\frac{z^2}{a^2\rho^2})^2}{\usup}_y{\usup}_z{\Big \}}{\Big
)}d\rho dz.
\end{eqnarray}

Replacing $m=3$ we get
$$
\hspace{-2.5cm}
{\ds
\frac{c_3Q_{\usup}(\xi_a)}{a^3}=\int_{\{-a\rho <z<a\rho \}}\rho^{4}(1-\frac{z^2}{a^2\rho^2})^2
{\Big
(}\eta_{\rho}^2{\usup}_z^2-}$$

$$-\eta^2{\Big
\{}\frac{4(1+\frac{z^2}{a^2\rho^2})}{\rho^2(1-\frac{z^2}{a^2\rho^2})^2}{\usup}_z^2
- \frac{8z}{a\rho^3(1-\frac{z^2}{a^2\rho^2})^2}{\usup}_y{\usup}_z{\Big \}}{\Big
)}d\rho dz.$$
We obtain, by using $\left(1-\frac{z^2}{a^2\rho^2}\right)^2\leq 1$ and 
$ 1+\frac{z^2}{a^2\rho^2}\geq 1$,

$$\frac{c_3Q_{\usup}(\xi_a)}{a^3}\leq \int_{\{-a\rho <z<a\rho \}}
\rho^{4}\usup_z^2(a\rho,z)\left(\eta_{\rho}^2-\frac{4}{\rho^2}\eta^2\right)d\rho
dz +$$

\[
+  \int_{\{-a\rho <z<a\rho \}}
\frac{8z\rho\eta^2(\rho)}{a}{\usup}_y(a\rho,z){\usup}_z(a\rho,z) d\rho
dz.
\]

We study these two integrals separately. From Theorem $\ref{asym}$
we have that $\usup_y(a\rho,z)\rightarrow 0$
uniformly, for all $\rho\in[\rho_1,\rho_2]=\textrm{supp }\eta$, as $a$ tends to infinity.  
Hence, given $\varepsilon>0$, for $a$ sufficiently 
large,
$|\usup_y(a\rho,z)|\leq \varepsilon$. Moreover, by Theorem
\ref{teominmax} we have that $\usup_z\geq 0$ in $\R^6$. Hence, since $\eta$
is bounded, for $a$ large we have
\begin{eqnarray*}
\left|\int \frac{8z\rho\eta^2(\rho)}{a}{\usup}_y{\usup}_z d\rho dz\right| & \leq & \int \left|\frac{8z\rho\eta^2(\rho)}{a}\right||{\usup}_y|{\usup}_z d\rho dz \\
 & \leq & \int 8\rho^2\eta^2(\rho)|\usup_y|\usup_z d\rho dz \\
& \leq & C\varepsilon \int_{\rho_1}^{\rho_2} \rho^2d\rho\int_{-a\rho}^{a\rho}\usup_zdz \\
& = & C\varepsilon \int_{\rho_1}^{\rho_2} \left(\usup(a\rho,a\rho)-\usup(a\rho,-a\rho)\right) d\rho \\
& \leq & C\varepsilon,
\end{eqnarray*}
where $C$ are different constants depending on $\rho_1$ and $\rho_2$. Hence, as
$a$ tends to infinity, this integral converges to zero. 

Now, for the other integral
we have that again by Theorem $\ref{asym}$, $\usup_z(a\rho,z)$ converges to
$\dot{u}_0(z)$. We write
\begin{eqnarray*}
\int_{\{-a\rho <z<a\rho \}}
\rho^{4}\usup_z^2(a\rho,z)\left(\eta_{\rho}^2-\frac{4}{\rho^2}\eta^2\right)d\rho
dz & = & \\
&  & \hspace{-19em}=\int_{\{-a\rho <z<a\rho \}} \dot{u}_0^2(z) \rho^{4}\left(\eta_{\rho}^2-\frac{4}{\rho^2}\eta^2\right)d\rho dz + \\
& & \hspace{-22em} + \int_{\{-a\rho <z<a\rho \}}
\rho^{4}(\usup_z(a\rho,z)-\dot{u}_0(z))(\usup_z(a\rho,z)+\dot{u}_0(z))\left(\eta_{\rho}^2-\frac{4}{\rho^2}\eta^2\right)d\rho dz.
\end{eqnarray*}
For $a$ large, $|\usup_z(a\rho,z)-\dot{u}_0(z)|\leq \varepsilon$ for all
$\rho\in [\rho_1,\rho_2]$. In addition,
$\usup_z(a\rho,z)+\dot{u}_0(z)$ is positive and is a derivative with
respect to $z$ of a bounded function, thus it is integrable in $z$
(by the fundamental theorem of calculus). 
Hence, since $\eta=\eta(\rho)$ is smooth with compact support, the
second integral converges to zero as $a$ tends to infinity.
Therefore, letting $a$ tend to infinity, we obtain
\begin{equation}\label{dim6}
\limsup_{a\rightarrow\infty}\frac{c_3Q_{\usup}(\xi_a)}{a^3}\leq \left(\int_{-\infty}^{+\infty} \dot{u}_0^2(z) dz\right)
\int \rho^{4}\left(\eta_{\rho}^2-\frac{4}{\rho^2}\eta^2\right)d\rho.
\end{equation}
Recall that the integral in $dz$ is finite (since
$\dot{u}_0^2\leq C\dot{u}_0$ and $\dot{u}_0$ is positive and is a derivative 
of a bounded function).

The integral in $d\rho$ can be seen as
an integral in $\R^5$ of radial functions $\eta=\eta(\rho)$. Hardy's inequality in
$\R^5$
states (see e.g.~\cite{GP}) that
$$\frac{(5-2)^2}{4}\int_{\R^5} \frac{\eta^2}{|x|^2}dx \leq \int_{\R^5}|\nabla
 \eta|^2dx $$ holds for every $H^1$ function $\eta=\eta(\rho)$ with
 compact support, and the  constant
 $(5-2)^2/4=9/4$ is the best possible (even among compactly supported
 Lipschitz functions). Hence, since
$$4>\frac{9}{4}=\frac{(5-2)^2}{4},$$ we have that the second integral
 in \eqref{dim6} is negative for some compactly supported Lipschitz
 function $\eta$.
We conclude that $\usup$ is unstable. Therefore, by the arguments
 preceding this proof, also every bounded solution
 of $-\Delta u=f(u)$ in $\R^6$ such that $u=0$ on the Simons cone
 ${\mathcal C}$ and $u$ is positive in ${\mathcal O}$ is unstable.

The following is a direct way (without using Hardy's inequality)  
to see that the second integral in \eqref{dim6} is negative 
for some Lipschitz function $\eta$ with 
compact support in $(0,\infty)$. 
For $\alpha >0$ and $0<2\rho_1<1<\rho_2$, let
\begin{equation*}
\eta(\rho)= \left\{
\begin{array}{ll}
(1-\rho_2^{-\alpha})\rho_1^{-1}(\rho-\rho_1) & \text{ for } \rho_1\leq \rho \leq 2\rho_1\\
1-\rho_2^{-\alpha} & \text{ for }2\rho_1\leq \rho \leq 1\\
\rho^{-\alpha}-\rho_2^{-\alpha} & \text{ for } 1\leq \rho \leq \rho_2\\
0 & \text{ otherwise},
\end{array}
\right.
\end{equation*}
a Lipschitz function with compact support $[\rho_1,\rho_2]$.
We simply compute the second integral in \eqref{dim6} and find
\begin{eqnarray*}
 \int_{0}^{+\infty}
 \rho^4\left\{\eta_{\rho}^2-\frac{4\eta^2}{\rho^2}\right\}d\rho & \leq & 
  \int_{\rho_1}^{2\rho_1}\rho^4 (1-\rho_2^{-\alpha})^2 \rho_1^{-2}
 d\rho+\\
& &  + 
\int_1^{+\infty} \alpha^2\rho^{2-2\alpha}d\rho
- \int_1^{\rho_2} 4\rho^2(\rho^{-\alpha}-\rho_2^{-\alpha})^2 d\rho.
\end{eqnarray*}
Choosing $3/2<\alpha<2$, as $\rho_2\to\infty$ the difference of the
last two integrals converges to a negative number,
since $\alpha^2<4$ and $2-2\alpha < -1$. 
Since the first of the three last integrals is
bounded by $7\rho_1^3$, we conclude that the above expression is
negative by choosing $\rho_2$ large enough and then $\rho_1$
small enough.

The previous proof of instability also leads to the instability
outside of every compact set ---and thus to the infinite Morse
index property of $u$. Indeed, choosing $\rho_1$ and
$\rho_2$ (and thus $\eta$) as above, we 
consider the corresponding function $\xi_a$ for $a>1$. Now,
\eqref{ineqQ} and \eqref{dim6} lead to 
$Q_u (\xi_a)\leq Q_{\usup}(\xi_a) <0$ for $a$ large enough. 
Thus, the Lipschitz function $\xi_{a}$ makes $Q_u$
negative for $a$ large, and has compact support contained in 
$\{a\rho_1 \leq (s+t)/\sqrt{2}\leq a\rho_2\}$. By approximation,
the same is true for a function $\xi$ of class $C^1$, not
only Lipschitz.
Hence, given any compact set $K$ of $\R^6$, by taking $a$ 
large enough we conclude that $u$ is unstable outside $K$,
as stated in the theorem.

{From} the instability outside every compact set, it follows
that $u$ has infinite Morse index in the sense of 
Definition~\ref{morse}. Indeed, let $X_k$ be a subspace
of $C^1_c(\R^6)$ of dimension $k$, generated by functions
$\xi_1,\ldots,\xi_k$, and with $Q_u(\xi)<0$ for all 
$\xi\in X_k\setminus\{0\}$.
Let $K$ be a compact set containing the support of all
the functions $\xi_1,\ldots,\xi_k$. Since $u$ is unstable
outside $K$, there is a $C^1$ function $\xi_{k+1}$ with 
compact support
in $\R^6\setminus K$ for which $Q_u(\xi_{k+1})<0$.
Since $\xi_{k+1}$ has disjoint support with each of the functions
$\xi_1,\ldots,\xi_k$, it follows that $\xi_1,\ldots,\xi_k,\xi_{k+1}$
are linearly independent and that $Q_u(a_1\xi_1+\cdots+a_{k+1}\xi_{k+1})
=Q_u(a_1\xi_1+\cdots+a_{k}\xi_{k})+Q_u(a_{k+1}\xi_{k+1})<0$ for every
nonzero linear combination $a_1\xi_1+\cdots+a_{k+1}\xi_{k+1}$ of them.
We conclude that $u$ has infinite Morse index.
\end{proof}

\section{Asymptotic stability of $\usup$ in dimensions $2m\geq 8$}

In \cite{CT} we proved that the argument used to prove instability in
dimension~$4$ could not be used in higher dimensions (see also the
begining of the previous section). It is natural
to check whether the computations we used in the proof of the
previous theorem in dimension 6 lead to a instability result in higher
dimensions. This is not the case, even in dimension $8$. This
is an indication that saddle solutions might be stable in
dimension 8 or higher.

We recall the computations in the proof of the previous theorem. Using
the linearized equation \eqref{linz} and
$\xi_a=\xi=\eta(y/a)\usup_z(y,z)=\eta(\rho)\usup_z(a\rho,z)$, we
reached in \eqref{dim62}
\[
{\ds
c_mQ_{\usup}(\xi_a)=a^{2m-3}\int_{\{-a\rho<z<a\rho\}}\rho^{2(m-1)}(1-\frac{z^2}{a^2\rho^2})^{m-1}{\Big
(}\eta_{\rho}^2\psi^2{\usup}_z ^2-}
\]
\[
-\eta^2{\Big
\{}\frac{2(m-1)(1+\frac{z^2}{a^2\rho^2})}{\rho^2(1-\frac{z^2}{a^2\rho^2})^2}{\usup}_z^2
-
\frac{4(m-1)z}{a\rho^3(1-\frac{z^2}{a^2\rho^2})^2}{\usup}_y{\usup}_z{\Big
\}}{\Big )}d\rho dz,
\]
where $\eta=\eta(\rho)$ has compact support
$[\rho_1,\rho_2]\subset (0,\infty)$ and $\usup_y,\usup_z$ are evaluated
  at $(a\rho,z)$. Dividing by $a^{2m-3}$, we obtain
\begin{eqnarray*}
&& \hspace{-.8cm}
\frac{c_mQ_{\usup}(\xi_a)}{a^{2m-3}} =  \\
&&
=\int_{\{-a\rho < z< a\rho\}}
\rho^{2(m-1)}\left(1-\frac{z^2}{a^2\rho^2}\right)^{m-1}
\usup_z^2 \cdot \\
&&\qquad\qquad\qquad\qquad
\cdot \left\{\eta_\rho^2-
\frac{2(m-1)}{\rho^2}\frac{(1+\frac{z^2}{a^2\rho^2})}{(1-\frac{z^2}{a^2\rho^2})^{2}}\eta^2
\right)d\rho dz \\
&&\qquad \quad +\int_{\{-a\rho< z < a\rho\}}
\frac{4(m-1)z}{a}\rho^{2m-5}(1-\frac{z^2}{a^2\rho^2})^{m-3}{\usup}_y{\usup}_zd\rho dz.
\end{eqnarray*}

Using $\usup_z\geq 0$ in $\R^{2m}$ and the asymptotic limit of $|\usup_y|$ as in the
proof of instability in dimension~6 (see the previous section), 
the integrals above lead to
$$\limsup_{a\rightarrow\infty}\frac{c_mQ_{\usup}(\xi_a)}{a^{2m-3}}= 
\left\{\int_{-\infty}^{+\infty}\dot{u}_0^2(z)dz\right\}\int\rho^{2(m-1)}\left( \eta_\rho^2
-\frac{2(m-1)}{\rho^2}\eta^2 \right)d\rho. $$ 
The second integral may be
thought as an integral of a radial function in $\R^{2m-1}$. Using
Hardy's inequality (see e.g.~\cite{GP})
we have that this integral is nonnegative for all Lipschitz $\eta$ if and
only if $$2(m-1)\leq \frac{(2m-1-2)^2}{4}.$$
Writting $n=2m$, the above inequality holds if and only if
$$n^2-10n+17\geq 0,$$
that is,  $n\geq 8$. Thus, for $n\geq 8$ the inequality is true (it is even strict) 
and we conclude
some kind of asymptotic stability of $\usup$.

\end{document}